\documentclass{article}
\usepackage[toc,page]{appendix}
\usepackage[utf8]{inputenc}
\usepackage[english]{babel}
\DeclareMathAlphabet{\mathscr}{OT1}{pzc}{m}{it} 
\usepackage{amsmath}
\usepackage{dsfont} 
\usepackage{amssymb}
\usepackage[T1]{fontenc}
\usepackage{mathtools}   
\usepackage{amsfonts}
\usepackage{bbm}
\usepackage{stmaryrd}
\usepackage{mathrsfs}  
\usepackage{graphicx}
\usepackage{bigints}
\usepackage{relsize}
\usepackage[colorinlistoftodos]{todonotes}
\usepackage{amssymb,amsmath,amsthm}
\numberwithin{equation}{section}
\usepackage{dsfont}
\usepackage[shortlabels]{enumitem}
\setlist{labelindent=\parindent,leftmargin=*}
\usepackage{authblk}

\newtheorem{theorem}{Theorem}[section]
\newtheorem{notation}[theorem]{Notation}
\newtheorem{lemma}[theorem]{Lemma}
\newtheorem{proposition}[theorem]{Proposition}
\newtheorem{corollary}[theorem]{Corollary}
\newtheorem{definition}[theorem]{Definition}

\newtheorem{remark}[theorem]{Remark}

\usepackage{amsmath,amssymb}
\makeatletter
\newsavebox\myboxA
\newsavebox\myboxB
\newlength\mylenA
\newcommand*\xoverline[2][0.75]{%
    \sbox{\myboxA}{$\m@th#2$}%
    \setbox\myboxB\null
    \ht\myboxB=\ht\myboxA%
    \dp\myboxB=\dp\myboxA%
    \wd\myboxB=#1\wd\myboxA
    \sbox\myboxB{$\m@th\overline{\copy\myboxB}$}
    \setlength\mylenA{\the\wd\myboxA}
    \addtolength\mylenA{-\the\wd\myboxB}%
    \ifdim\wd\myboxB<\wd\myboxA%
       \rlap{\hskip 0.5\mylenA\usebox\myboxB}{\usebox\myboxA}%
    \else
        \hskip -0.5\mylenA\rlap{\usebox\myboxA}{\hskip 0.5\mylenA\usebox\myboxB}%
    \fi}
\makeatother

\title{A note on  time-dependent additive functionals}

\author{
Adrien BARRASSO \thanks{ENSTA ParisTech, Unit\'e de Math\'ematiques
 appliqu\'ees, 828, boulevard des Mar\'echaux, F-91120 Palaiseau, France 
 and Ecole Polytechnique,  F-91128 Palaiseau, France.
E-mail: {\sf adrien.barrasso@ensta-paristech.fr}. \\
 This author was supported by a PhD fellowship (AMX) of the Ecole Polytechnique.}
\qquad\quad
Francesco RUSSO\thanks{ENSTA ParisTech, Unit\'e de Math\'ematiques appliqu\'ees, 828, boulevard des Mar\'echaux, F-91120 Palaiseau, France. E-mail: {\sf
 francesco.russo@ensta-paristech.fr.} \\
The financial support of this author was partially provided
 by the DFG through the CRC ''Taming uncertainty and profiting from 
randomness and low regularity in analysis, stochastics and their application''.}
}

\date{August 2nd 2017}

\begin{document}
\maketitle

{\bf Abstract.} This note develops shortly the theory of
time-inhomogeneous additive functionals and 
is a useful support for the analysis of time-dependent
Markov processes and related topics. It is 
a significant tool for the analysis of BSDEs in law.
In particular we extend to a non-homogeneous setup some results
 concerning the quadratic variation and the angular bracket of 
 Martingale Additive Functionals (in short MAF) associated to a homogeneous Markov processes.



\section{Introduction}

The notion of Additive  Functional of a general  Markov process is due to 
E.B Dynkin and has been studied since the early '60s by the Russian, 
French and American schools of probability, see for example
 \cite{dynkin1959foundations}, \cite{MeyerPhdAF}, \cite{BlumGetAF}. A mature version of the homogeneous theory may be found for example in \cite{dellmeyerD},
 Chapter XV. In that context, given a  probability $\mu$ on some state space $E$,  
  $\mathbbm{P}^{\mu}$ denotes the law of a time-homogeneous Markov process 
with initial law $\mu$.
 An {\bf Additive Functional} is a  process $(A_t)_{t\geq 0}$
defined on a canonical space, adapted to the
canonical filtration  such that for any
 $s\leq t$ and $\mu$,  $A_{s+t}=A_s+A_t\circ \theta_s$ $\mathbbm{P}^{\mu}$-a.s., 
where $\theta$ is the usual shift operator on the canonical space.

 If moreover $A$ is under any law  $\mathbbm{P}^{\mu}$ a martingale, then it 
is called a Martingale Additive Functional (MAF). The quadratic variation 
and angular  bracket of a MAF were shown to be AFs in \cite{dellmeyerD}.
We  extend this type of results to a more general definition of an AF which
 is closer to the original notion of Additive Functional associated to a 
stochastic system introduced by E.B. Dynkin, see \cite{dynkin1975additive} for instance.
\\
\\
Our setup will be the following. We consider a canonical Markov class 
$(\mathbbm{P}^{s,x})_{(s,x)\in[0,T]\times E}$ with time index $[0,T]$ and state
 space $E$ being a Polish space. For any $(s,x)\in[0,T]\times E$, 
$\mathbbm{P}^{s,x}$ corresponds to the probability law 
(defined on some canonical filtered space $\left(\Omega,\mathcal{F},(\mathcal{F}_t)_{t\in[0,T]}\right)$) of a Markov process starting from point $x$ at time $s$.
On $(\Omega,\mathcal{F})$, we define a \textbf{non-homogeneous 
Additive Functional} (shortened by AF) as a real-valued random-field 
$A:=(A^t_u)_{0\leq t\leq u\leq T}$ verifying the two following conditions.
\begin{enumerate}
	\item For any $0\leq t\leq u\leq T$, $A^t_u$ is $\mathcal{F}_{t,u}$-measurable;
	\item for any $(s,x)\in[0,T]\times E$, there exists a real cadlag $(\mathcal{F}^{s,x}_t)_{t\in[0,T]}$-adapted process $A^{s,x}$ (taken equal to zero on $[0,s]$ by convention) such that for any $x\in E$ and $s\leq t\leq u$, $A^t_u = A^{s,x}_u-A^{s,x}_t \,\text{  }\, \mathbbm{P}^{s,x}$ a.s.
\end{enumerate}
Where $\mathcal{F}_{t,u}$ is the $\sigma$-field generated by the canonical process between time $t$ and $u$, and $\mathcal{F}^{s,x}_t$ is obtained by adding the $\mathbbm{P}^{s,x}$ negligible sets to $\mathcal{F}_t$.
$A^{s,x}$ will be called the \textbf{cadlag version of $A$ under} $\mathbbm{P}^{s,x}$. If for any $(s,x)$, $A^{s,x}$ is a $(\mathbbm{P}^{s,x},(\mathcal{F}_t)_{t\in[0,T]})$-square integrable martingale then $A$ will be called a square integrable Martingale Additive Functional (in short, square integrable MAF).
\\
\\
The main contributions  of the paper are essentially the following. In Section \ref{A1}, we recall the definition and prove some basic results concerning canonical Markov classes. In Section \ref{A2}, we start by defining an AF in Definition \ref{DefAF}. In Proposition \ref{VarQuadAF}, we show that if  $(M^t_u)_{0\leq t\leq u\leq T}$ is a square integrable MAF, then there exists an AF $([M]^t_u)_{0\leq t\leq u\leq T}$ which  for any $(s,x)\in[0,T]\times E$, has  $[M^{s,x}]$
as cadlag version under $\mathbbm{P}^{s,x}$. Corollary \ref{AFbracket} states that given two square integrable MAFs $(M^t_u)_{0\leq t\leq u\leq T}$, $(N^t_u)_{0\leq t\leq u\leq T}$, there exists an AF, denoted by
 $(\langle M,N\rangle^t_u)_{0\leq t\leq u\leq T}$, which has  $\langle M^{s,x},N^{s,x}\rangle$ as cadlag version under $\mathbbm{P}^{s,x}$. Finally, we prove in Proposition \ref{BracketMAFnew} that if $M$ or $N$
is such that for
 $\mathbbm{P}^{s,x}$, its cadlag version under $\mathbbm{P}^{s,x}$,
 its angular bracket is absolutely continuous with respect to some continuous non-decreasing function $V$, then there exists a Borel function $v$ such that for any $(s,x)$, $\langle M^{s,x},N^{s,x}\rangle=\int_s^{\cdot\vee s}v(r,X_r)dV_r$. \\
The present note constitutes a support for the
authors, in the analysis of deterministic
problems related to Markovian type backward stochastic differential
equations where the forward process is given in law, see e.g.
\cite{paper1preprint}. Indeed, when the forward process of the BSDE does 
not define a stochastic flow (typically if it is not the strong solution of
 an SDE but only a weak solution), we cannot exploit the mentioned 
flow property to show that the solution of the BSDE is a function of 
the forward process, as it is usually done, see Remark 5.35 (ii) in \cite{PardouxRascanu} for instance.

\section{Preliminaries}

\label{SPrelim}

In the whole paper we will use the following notions, notations and vocabulary.
\\
\\
A topological space $E$ will always be considered as a measurable space with its Borel $\sigma$-field which shall be denoted $\mathcal{B}(E)$ and if $S$
is another topological space equipped
with its Borel $\sigma$-field, $\mathcal{B}(E,S)$  will denote the set of Borel functions from $E$ to $S$.
\\

Let $(\Omega,\mathcal{F})$, $(E,\mathcal{E})$ be two measurable spaces. A measurable mapping from $(\Omega,\mathcal{F})$ to $(E,\mathcal{E})$ shall often be called a \textbf{random variable} (with values in $E$), or in short r.v. If $\mathbbm{T}$ is some set, an indexed set of r.v. with values in $E$, $(X_t)_{t\in \mathbbm{T}}$ will be called a \textbf{random field} (indexed by $\mathbbm{T}$ with values in $E$). In particular, if $\mathbbm{T}$ is an interval included in $\mathbbm{R}_+$, $(X_t)_{t\in \mathbbm{T}}$ will be called a \textbf{stochastic process} (indexed by $\mathbbm{T}$ with values in $E$). Given a stochastic process, if the mapping 
\begin{equation*}
\begin{array}{rcl}
(t,\omega)&\longmapsto&X_t(\omega)\\
(\mathbbm{T}\times\Omega,\mathcal{B}(\mathbbm{T})\otimes\mathcal{F})&
\longrightarrow&(E,\mathcal{E})
\end{array}
\end{equation*}
is measurable, then the process $(X_t)_{t\in \mathbbm{T}}$ will be called a \textbf{measurable process} (indexed by $\mathbbm{T}$ with values in $E$).
\\
\\
On a fixed probability space $\left(\Omega,\mathcal{F},\mathbbm{P}\right)$, for any $p \ge 1$, $L^p$ will denote the set of random variables with finite $p$-th moment.
Two random fields (or stochastic processes) $(X_t)_{t\in \mathbbm{T}}$, $(Y_t)_{t\in \mathbbm{T}}$ indexed by the same set and with values in the same space will be said to be \textbf{modifications (or versions)
 of each other} if for every $t\in\mathbbm{T}$, $\mathbbm{P}(X_t=Y_t)=1$.
\\
\\
A probability space equipped with a right-continuous filtration
 $\left(\Omega,\mathcal{F},(\mathcal{F}_t)_{t\in\mathbbm{T}},\mathbbm{P}\right)$  will be called called a \textbf{stochastic basis} and will be said to \textbf{fulfill the usual conditions} if the probability space is complete and if $\mathcal{F}_0$ contains all the $\mathbbm{P}$-negligible sets.
\\
\\
Concerning spaces of stochastic processes, in a fixed stochastic basis $\left(\Omega,\mathcal{F},(\mathcal{F}_t)_{t\in\mathbbm{T}},\mathbbm{P}\right)$, we will use the following notations and vocabulary.
$\mathcal{M}$ will be the space of cadlag martingales.  
\\
For any $p\in[1,\infty]$  $\mathcal{H}^p$ will denote
the subset of $\mathcal{M}$ of elements $M$ such that $\underset{t\in \mathbbm{T}}{\text{sup }}|M_t|\in L^p$ and in this set we identify indistinguishable elements. It is a Banach space for  the norm
$\| M\|_{\mathcal{H}^p}=\mathbbm{E}[|\underset{t\in \mathbbm{T}}{\text{sup }}M_t|^p]^{\frac{1}{p}}$, and
$\mathcal{H}^p_0$ will denote the Banach subspace of $\mathcal{H}^p$
containing the elements starting at zero.
\\
If $\mathbbm{T}=[0,T]$ for some $T\in\mathbbm{R}_+^*$, a stopping time will be defined as a random variable with values in 
$[0,T]\cup\{+\infty\}$ such that for any $t\in[0,T]$, $\{\tau\leq t\}\in \mathcal{F}_t$. We define a \textbf{localizing sequence of stopping times} as an increasing sequence of stopping times $(\tau_n)_{n\geq 0}$ such that there exists $N\in\mathbbm{N}$ for which $\tau_N=+\infty$. Let $Y$ be a process and $\tau$ a stopping time, we denote  $Y^{\tau}$ the process $t\mapsto Y_{t\wedge\tau}$ which we call \textbf{stopped process}.  If $\mathcal{C}$ is a set of processes, we define its \textbf{localized class} $\mathcal{C}_{loc}$ as the set of processes $Y$ such that there exist a localizing sequence $(\tau_n)_{n\geq 0}$ such that for every $n$, the stopped process $Y^{\tau_n}$ belongs to $\mathcal{C}$.
\\
For any $M,N\in  \mathcal{M}$, we denote $[M]$ (resp. $[M,N]$) the \textbf{quadratic variation} of $M$ (resp. \textbf{quadratic covariation} of $M,N$). If $M,N\in\mathcal{H}_{loc}^2$, $\langle M,N\rangle$ (or simply $\langle M\rangle$ if $M=N$) will denote their (predictable) \textbf{angular bracket}.
$\mathcal{H}_0^2$ will be equipped with scalar product defined by $(M,N)_{\mathcal{H}^2}:=\mathbbm{E}[M_TN_T] 
=\mathbbm{E}[\langle M, N\rangle_T] $ which makes it a Hilbert space.
Two elements $M,N$ of $\mathcal{H}^2_{0,loc}$ will be said to be \textbf{strongly orthogonal} if  $\langle M,N\rangle=0$.
\\
If $A$ is an adapted process with bounded variation then $Var(A)$ (resp. 
$Pos(A)$,  $Neg(A)$) will denote its total variation (resp. positive variation,
  negative variation), see Proposition 3.1, chap. 1 in \cite{jacod}.
In particular for almost all $\omega \in \Omega$, 
$t\mapsto Var_t(A(\omega))$ is the total variation function of
the function $t\mapsto A_t(\omega)$.

\section{Markov classes}\label{A1}

We recall here some basic definitions and results concerning Markov processes. For a complete study of homogeneous Markov processes, one may consult \cite{dellmeyerD}, concerning non-homogeneous Markov classes, our reference was Chapter VI of \cite{dynkin1982markov}. 

\subsection{Definition and basic results}

The first definition refers to the  canonical space that one can find in \cite{jacod79}, see paragraph 12.63.
\begin{notation}\label{canonicalspace}
In the whole section  $E$ will be a fixed  Polish  space (a separable completely metrizable topological space), and $\mathcal{B}(E)$  its Borel $\sigma$-field. $E$ will be called the \textbf{state space}. 
\\
\\ 
We consider $T\in\mathbbm{R}^*_+$. We denote $\Omega:=\mathbbm{D}(E)$ the Skorokhod space of functions from $[0,T]$ to $E$  right-continuous  with left limits and continuous at time $T$ (e.g. cadlag). For any $t\in[0,T]$ we denote the coordinate mapping $X_t:\omega\mapsto\omega(t)$, and we introduce on $\Omega$ the $\sigma$-field  $\mathcal{F}:=\sigma(X_r|r\in[0,T])$.

\begin{remark}
	All the results of the present paper remain valid if $\Omega$ is the space of continuous functions from $[0,T]$ to $E$, and to a time index equal to $\mathbbm{R}_+$.
\end{remark}

On the measurable space $(\Omega,\mathcal{F})$, we introduce the \textbf{canonical process}
\begin{equation*}
X:
\begin{array}{rcl}
(t,\omega)&\longmapsto& \omega(t)\\ \relax
([0,T]\times \Omega,\mathcal{B}([0,T])\otimes\mathcal{F}) &\longrightarrow & (E,\mathcal{B}(E)),
\end{array}
\end{equation*}
and the right-continuous filtration $(\mathcal{F}_t)_{t\in[0,T]}$ where $\mathcal{F}_t:=\underset{s\in]t,T]}{\bigcap}\sigma(X_r|r\leq s)$ if $t<T$, and $\mathcal{F}_T:= \sigma(X_r|r\in[0,T])=\mathcal{F}$.
\\
\\
$\left(\Omega,\mathcal{F},(\mathcal{F}_t)_{t\in[0,T]}\right)$ will be called the \textbf{canonical space} (associated to $T$ and $E$).
\\
\\
For any $t \in [0,T]$ we denote $\mathcal{F}_{t,T}:=\sigma(X_r|r\geq t)$, and
for any $0\leq t\leq u<T$ we will denote
$\mathcal{F}_{t,u}:= \underset{n\geq 0}{\bigcap}\sigma(X_r|r\in[t,u+\frac{1}{n}])$.
\end{notation}
We recall that since $E$ is Polish, then $\mathbbm{D}(E)$ can be equipped with a Skorokhod distance which makes it a Polish metric space (see Theorem 5.6 in Chapter 3 of \cite{EthierKurz}), and for which the Borel $\sigma$-field is $\mathcal{F}$ (see Proposition 7.1 in Chapter 3 of \cite{EthierKurz}). This in particular implies that $\mathcal{F}$ is separable, as the Borel $\sigma$-field of a separable metric space.

\begin{remark}\label{RemFiltr}
The above $\sigma$-fields fulfill the properties below.
\begin{enumerate}
\item For any $0\leq t\leq u < T$, $\mathcal{F}_{t,u}=\mathcal{F}_u\cap \mathcal{F}_{t,T}$; 
\item for any $t\geq 0$, $\mathcal{F}_t \vee \mathcal{F}_{t,T} = \mathcal{F}$;
\item for any $(s,x)\in[0,T]\times E$, the two first items remain true when considering the $\mathbbm{P}^{s,x}$-closures of all the $\sigma$-fields; 
\item for any $t\geq 0$, $\Pi:=\{F=F_t\cap F^t_{T}| (F_t,F^t_{T}) \in \mathcal{F}_t\times  \mathcal{F}_{t,T}\}$ is a $\pi$-system generating $\mathcal{F}$, i.e.
it is stable with respect to the intersection.
\end{enumerate}
\end{remark}


\begin{definition}\label{Defp}
	The function 
	\begin{equation*}
	P:\begin{array}{rcl}
	(s,t,x,A) &\longmapsto& P_{s,t}(x,A)   \\ \relax
	[0,T]^2\times E\times\mathcal{B}(E) &\longrightarrow& [0,1], 
	\end{array}
	\end{equation*}
	will be called \textbf{transition kernel} if, for any $s,t$ in $[0,T]$, $x\in E$,  $A\in \mathcal{B}(E)$, it verifies the following.
	
	\begin{enumerate}
		\item $P_{s,t}(\cdot,A)$ is Borel,
		\item $P_{s,t}(x,\cdot)$ is a probability measure on $(E,\mathcal{B}(E))$,
		\item if $t\leq s$ then $P_{s,t}(x,A)=\mathds{1}_A(x)$,
		\item if $s<t$, for any $u>t$, $\int_{E} P_{s,t}(x,dy)P_{t,u}(y,A) = P_{s,u}(x,A)$.
	\end{enumerate}
\end{definition}
The latter statement is the well-known \textbf{Chapman-Kolmogorov equation}.

\begin{definition}\label{DefFoncTrans}
	A transition kernel $P$ for which  the first item is reinforced 
	supposing that $(s,x)\longmapsto P_{s,t}(x,A)$ is Borel for any $t,A$,
	will be said to be \textbf{measurable in time}.
	
\end{definition}

\begin{remark} \label{RDefFoncTrans}
 Let $P$ be a transition kernel which is measurable in time.
By approximation by 
simple functions, one can easily show that,
  for any Borel function $\phi$ from $E$ to $\mathbbm{R}$ then
$(s,x)\mapsto \int \phi(y)P_{s,t}(x,dy)$ is Borel, provided
previous integral  makes sense.
In this paper we will only consider transition kernels which are measurable in time.
\end{remark}

\begin{definition}\label{defMarkov}
A \textbf{canonical Markov class} associated to a transition kernel $P$ is a set of probability measures $(\mathbbm{P}^{s,x})_{(s,x)\in[0,T]\times E}$ defined on the measurable space 
$(\Omega,\mathcal{F})$ and verifying for any $t \in [0,T]$ and $A\in\mathcal{B}(E)$
\begin{equation}\label{Markov1}
\mathbbm{P}^{s,x}(X_t\in A)=P_{s,t}(x,A),
\end{equation}
and for any $s\leq t\leq u$
\begin{equation}\label{Markov2}
\mathbbm{P}^{s,x}(X_u\in A|\mathcal{F}_t)=P_{t,u}(X_t,A)\quad \mathbbm{P}^{s,x}\text{ a.s.}
\end{equation}
\end{definition}
\begin{remark}\label{Rfuturefiltration}
	Formula 1.7 in Chapter 6 of \cite{dynkin1982markov} states
	that for $(s,x)\in[0,T]\times E$, $t\geq s$ and $F\in \mathcal{F}_{t,T}$ yields
	\begin{equation}\label{Markov3}
	\mathbbm{P}^{s,x}(F|\mathcal{F}_t) = \mathbbm{P}^{t,X_t}(F)=\mathbbm{P}^{s,x}(F|X_t)\,\text{  }\,  \mathbbm{P}^{s,x} \text{a.s.}
	\end{equation}
	Property  \eqref{Markov3}  will  be called 
	\textbf{Markov property}.
\end{remark}
For the rest of this section, we are given a canonical Markov class $(\mathbbm{P}^{s,x})_{(s,x)\in[0,T]\times E}$ which transition kernel is measurable in time.

Proposition A.10 in \cite{paper2} states the following.
\begin{proposition}\label{Borel}
For any event $F\in \mathcal{F}$,  
$(s,x)\longmapsto \mathbbm{P}^{s,x}(F)$ is Borel.
For any random variable $Z$, if the function $(s,x)\longmapsto \mathbbm{E}^{s,x}[Z]$ 
is well-defined (with possible values in $[-\infty, \infty]$),
then it is Borel. 
\end{proposition}

\begin{definition}\label{CompletedBasis}
For any $(s,x)\in[0,T]\times E$ we will consider the  $(s,x)$-\textbf{completion} $\left(\Omega,\mathcal{F}^{s,x},(\mathcal{F}^{s,x}_t)_{t\in[0,T]},\mathbbm{P}^{s,x}\right)$ of the stochastic basis $\left(\Omega,\mathcal{F},(\mathcal{F}_t)_{t\in[0,T]},\mathbbm{P}^{s,x}\right)$ by defining $\mathcal{F}^{s,x}$ as  the $\mathbbm{P}^{s,x}$-completion of $\mathcal{F}$ , by extending $\mathbbm{P}^{s,x}$ to $\mathcal{F}^{s,x}$ and finally by defining  $\mathcal{F}^{s,x}_t$ as the $\mathbbm{P}^{s,x}$-closure of $\mathcal{F}_t$, for every $t\in[0,T]$. 
\end{definition}

We remark that, for any $(s,x)\in[0,T]\times E$, $\left(\Omega,\mathcal{F}^{s,x},(\mathcal{F}^{s,x}_t)_{t\in[0,T]},\mathbbm{P}^{s,x}\right)$ is a stochastic basis fulfilling the usual conditions, see 1.4 in \cite{jacod} Chapter I. 

We recall the following simple consequence of Remark 32 in \cite{dellmeyer75} Chapter II.
\begin{proposition}\label{Fversion}
	Let $\mathcal{G}$ be a sub-$\sigma$-field of $\mathcal{F}$, $\mathbbm{P}$ a probability on $(\Omega,\mathcal{F})$ and $\mathcal{G}^{\mathbbm{P}}$ the $\mathbbm{P}$-closure of $\mathcal{G}$. Let $Z^{\mathbbm{P}}$ be a real $\mathcal{G}^{\mathbbm{P}}$-measurable random variable. There exists a $\mathcal{G}$-measurable random variable $Z$ such that $Z=Z^{\mathbbm{P}}$ $\mathbbm{P}$-a.s.
\end{proposition}

From this we can deduce the following.
\begin{proposition}\label{ConditionalExp} Let $(s,x)\in[0,T]\times E$ be fixed, $Z$ be a random variable and $t\in[s,T]$. Then 
$\mathbbm{E}^{s,x}[Z|\mathcal{F}_t]=\mathbbm{E}^{s,x}[Z|\mathcal{F}^{s,x}_t]$ $\mathbbm{P}^{s,x}$ a.s.
\end{proposition}
\begin{proof}
	$\mathbbm{E}^{s,x}[Z|\mathcal{F}_t]$ is $\mathcal{F}_t$-measurable and therefore $\mathcal{F}^{s,x}_t$-measurable. Moreover, let $G^{s,x}\in\mathcal{F}^{s,x}_t$, by Remark 32 in \cite{dellmeyer75} Chapter II, there exists $G\in\mathcal{F}_t$ such that 
	\\
	$\mathbbm{P}^{s,x}(G\cup G^{s,x})=\mathbbm{P}^{s,x}(G\backslash G^{s,x})$ implying $\mathds{1}_G=\mathds{1}_{G^{s,x}}$ $\mathbbm{P}^{s,x}$ a.s. So 
	\begin{equation*}
	\begin{array}{rcl}
	\mathbbm{E}^{s,x}\left[\mathds{1}_{G^{s,x}}\mathbbm{E}^{s,x}[Z|\mathcal{F}_t]\right]&=&\mathbbm{E}^{s,x}\left[\mathds{1}_G\mathbbm{E}^{s,x}[Z|\mathcal{F}_t]\right]\\
	&=& \mathbbm{E}^{s,x}\left[\mathds{1}_GZ\right]\\
	&=& \mathbbm{E}^{s,x}\left[\mathds{1}_{G^{s,x}}Z\right],
	\end{array}
	\end{equation*}
	where the second equality occurs  because of the definition of $\mathbbm{E}^{s,x}[Z|\mathcal{F}_t]$.
\end{proof}
In particular, under the probability $\mathbbm{P}^{s,x}$, $(\mathcal{F}_t)_{t\in[0,T]}$-martingales and $(\mathcal{F}^{s,x}_t)_{t\in[0,T]}$-martingales coincide.
\\
\\
We now show that in our setup, a 
canonical 
Markov class verifies the \textbf{Blumenthal 0-1 law} in the following sense.
\begin{proposition}\label{blumenthal}
Let $(s,x)\in[0,T]\times E$ and $F\in\mathcal{F}_{s,s}$. 
Then $\mathbbm{P}^{s,x}(F)$ is equal to $1$ or to $0$;
In other words, $\mathcal{F}_{s,s}$ is $\mathbbm{P}^{s,x}$-trivial.
\end{proposition}
\begin{proof}
Let $F\in\mathcal{F}_{s,s}$ as 
introduced  in Notation \ref{canonicalspace}.
\\
Since by Remark \ref{RemFiltr}, 
$\mathcal{F}_{s,s}=\mathcal{F}_s\cap\mathcal{F}_{s,T}$, then $F$ belongs to $\mathcal{F}_s$ so by conditioning we get 
\begin{equation*}
\begin{array}{rcl}
\mathbbm{E}^{s,x}[\mathds{1}_F] &=& \mathbbm{E}^{s,x}[\mathds{1}_F\mathds{1}_F]\\
&=&\mathbbm{E}^{s,x}[\mathds{1}_F\mathbbm{E}^{s,x}[\mathds{1}_F|\mathcal{F}_s]]\\
&=& \mathbbm{E}^{s,x}[\mathds{1}_F\mathbbm{E}^{s,X_s}[\mathds{1}_F]],
\end{array}
\end{equation*} 
where the latter  equality comes from \eqref{Markov3} because 
$F\in\mathcal{F}_{s,T}$.
But  $X_s=x$, $\mathbbm{P}^{s,x}$ a.s., so 

\begin{equation*}
\begin{array}{rcl}
\mathbbm{E}^{s,x}[\mathds{1}_F] &=&\mathbbm{E}^{s,x}[\mathds{1}_F\mathbbm{E}^{s,x}[\mathds{1}_F]]\\
&=&\mathbbm{E}^{s,x}[\mathds{1}_F]^2.
\end{array}
\end{equation*}
\end{proof}

\subsection{Examples of canonical Markov classes}
 
We will list here some well-known examples of canonical Markov classes and some more recent ones.
\begin{itemize}
\item Let $E:=\mathbbm{R}^d$ for some $d\in\mathbbm{N}^*$.
We are given 
$b\in\mathcal{B}_b(\mathbbm{R}_+\times \mathbbm{R}^d, \mathbbm{R}^d)$, $\alpha\in\mathcal{C}_b(\mathbbm{R}_+\times \mathbbm{R}^d,S^*_+(\mathbbm{R}^d))$ 
(where $S^*_+(\mathbbm{R}^d)$ is the space of symmetric strictly positive definite matrices of size $d$) and $K$ a L\'evy kernel (this means  that for every $(t,x)\in \mathbbm{R}_+\times \mathbbm{R}^d$, $K(t,x,\cdot)$ is a $\sigma$-finite measure 
on $\mathbbm{R}^d\backslash\{0\}$, $\underset{t,x}{\text{sup}}\int \frac{\|y\|^2}{1+\|y\|^2}K(t,x,dy)<\infty$ and for every Borel set $A\in\mathcal{B}(\mathbbm{R}^d\backslash\{0\})$, 
$(t,x)\longmapsto \int_A \frac{\|y\|^2}{1+\|y\|^2}K(t,x,dy)$ is Borel) such that for any $A\in\mathcal{B}(\mathbbm{R}^d\backslash\{0\})$, 
$(t,x)\longmapsto \int_A \frac{y}{1+\|y\|^2}K(t,x,dy)$ is bounded continuous.
\\
Let $a$ denote the operator defined on some  $\phi\in\mathcal{C}^{1,2}_b(\mathbbm{R}_+\times\mathbbm{R}^d)$ by 
\begin{equation}
	\partial_t\phi + \frac{1}{2}Tr(\alpha\nabla^2\phi) + (b,\nabla \phi) +\int\left(\phi(\cdot,\cdot+y)-\phi-\frac{(y,\nabla \phi)}{1+\|y\|^2}\right)K(\cdot,\cdot,dy)
\end{equation}
In \cite{stroock1975diffusion} (see Theorem 4.3 and the penultimate sentence of its proof), the following is shown. \\
For every $(s,x)\in \mathbbm{R}_+\times \mathbbm{R}^d$, there exists a unique probability $\mathbbm{P}^{s,x}$ on the canonical space (see Definition \ref{canonicalspace}) such that $\phi(\cdot,X_{\cdot})-\int_s^{\cdot}a(\phi)(r,X_r)dr$ is a local martingale for every
 $\phi\in\mathcal{C}^{1,2}_b(\mathbbm{R}_+\times\mathbbm{R}^d)$ and $\mathbbm{P}^{s,x}(X_s=x)=1$. 
Moreover  $(\mathbbm{P}^{s,x})_{(s,x)\in\mathbbm{R}_+\times\mathbbm{R}^d}$ defines a canonical Markov class and its transition kernel is measurable in time.
\\
\item The case $K=0$ was studied extensively in the celebrated
 book \cite{stroock} in which it is also shown that if $b$, $\alpha$ are bounded and continuous in the second variable, then there exists a canonical Markov class with transition kernel  measurable in time $(\mathbbm{P}^{s,x})_{(s,x)\in\mathbbm{R}_+\times\mathbbm{R}^d}$ such that  $\phi(\cdot,X_{\cdot})-\int_s^{\cdot}a(\phi)(r,X_r)dr$ is a local martingale for any $\phi\in\mathcal{C}^{1,2}_b(\mathbbm{R}_+\times\mathbbm{R}^d)$.
\\
\item In \cite{rozkosz}, a canonical Markov class whose transition kernel
 is the weak fundamental solution of a parabolic PDE in divergence form 
is exhibited.
\\
\item In \cite{hsu}, diffusions on manifolds are studied and shown to define canonical Markov classes.
\\
\item Solutions of PDEs with distributional drift are exhibited in \cite{frw1} and shown to define canonical Markov classes.
\end{itemize}
Some of  previous examples were only studied as homogeneous Markov 
processes but can easily be shown to fall in the non-homogeneous setup 
of the present paper as it was illustrated in \cite{paper2}.

\section{Martingale Additive Functionals}\label{A2}

We now introduce the notion of non-homogeneous Additive 
Functional that we use in the paper. This looks to be a good compromise
 between the notion of Additive
 Functional associated to a stochastic system introduced by E.B. Dynkin 
(see for example \cite{dynkin1975additive}) and the more popular notion of 
homogeneous Additive Functional studied extensively, for instance by C. Dellacherie and P.A. Meyer in \cite{dellmeyerD} Chapter XV. This section 
 consists in extending some essential results stated in
 \cite{dellmeyerD} Chapter XV to our setup. 

\begin{definition}\label{DefAF} We denote $\Delta:=\{(t,u)\in[0,T]^2|t\leq u\}$.
	On $(\Omega,\mathcal{F})$, we define a 
\textbf{non-homogeneous Additive Functional} (shortened AF) as a random-field
$A:=(A^t_u)_{(t,u)\in\Delta}$ 
 indexed by $\Delta$ with values in $\mathbbm{R},$  
	verifying the two following conditions.
	\begin{enumerate}
		\item For any $(t,u)\in\Delta$, $A^t_u$ is $\mathcal{F}_{t,u}$-measurable;
		\item for any $(s,x)\in[0,T]\times E$, there exists a real cadlag $\mathcal{F}^{s,x}$-adapted process $A^{s,x}$ (taken equal to zero on $[0,s]$ by convention) such that for any $x\in E$ and $s\leq t\leq u$, $A^t_u = A^{s,x}_u-A^{s,x}_t \,\text{  }\, \mathbbm{P}^{s,x}$ a.s.
	\end{enumerate}
	$A^{s,x}$ will be called the \textbf{cadlag version of $A$ under} $\mathbbm{P}^{s,x}$.
	\\
	\\
	An AF will be called a \textbf{non-homogeneous square integrable Martingale Additive Functional} (shortened square integrable MAF) if under any $\mathbbm{P}^{s,x}$ its cadlag version is a square integrable martingale.
	More generally an AF will be said to verify a certain property 
	(being non-negative, increasing, of bounded variation, square integrable,
	having  $L^1$ terminal value) if under any $\mathbbm{P}^{s,x}$ its cadlag version verifies it.
	\\
	\\
	Finally, given an increasing AF $A$ and an increasing function $V$, $A$ will be said to be
	 \textbf{absolutely continuous with respect to} $V$ if for any 
	$(s,x)\in [0,T]\times E$, $dA^{s,x} \ll dV$ in the sense of stochastic measures.
\end{definition}


In this section for a given MAF $(M^t_u)_{(t,u)\in\Delta}$  we will be able to exhibit
 two AF, denoted respectively by  $([M]^t_u)_{(t,u)\in\Delta}$ 
and    $(\langle M \rangle^t_u)_{(t,u)\in\Delta}$,
which will play respectively 
the role of a quadratic variation and an angular bracket of it.
Moreover we will show 
 that the Radon-Nikodym derivative of the mentioned angular bracket of 
a MAF with respect to our reference function $V$ is 
a time-dependent function of the underlying process.

\begin{proposition}\label{VarQuadAF} 
Let $(M^t_u)_{(t,u)\in\Delta}$ be a square integrable MAF, and for any $(s,x)\in[0,T]\times E$,
  $[M^{s,x}]$ be the quadratic variation of its cadlag version $M^{s,x}$ under $\mathbbm{P}^{s,x}$.
 Then there exists an AF which we will call $([M]^t_u)_{(t,u)\in\Delta}$ and which, for any $(s,x)\in[0,T]\times E$, has  $[M^{s,x}]$
as cadlag version
 under $\mathbbm{P}^{s,x}$. 
\end{proposition}

\begin{proof}
We adapt Theorem 16 Chapter XV in \cite{dellmeyerD} to a non homogeneous 
set-up but the reader must keep in mind that our definition of 
Additive Functional is different from the one related to the
 homogeneous case.
\\
\\
For the whole proof $t<u$ will be fixed. 
 We consider a sequence of subdivisions of $[t,u]$: $t=t^k_1<t^k_2<\cdots<t^k_k=u$ such that $\underset{i<k}{\text{min }}(t^k_{i+1}-t^k_i)
\underset{k\rightarrow \infty}{\longrightarrow} 0$.
Let $(s,x)\in[0,t]\times E$  with corresponding
probability $\mathbbm{P}^{s,x}$.
For any $k$, we have $\underset{i< k}{\sum}\left( M_{t^k_{i+1}}^{t^k_i}\right)^2=\underset{i< k}{\sum}(M^{s,x}_{t^k_{i+1}}-M^{s,x}_{t^k_i})^2$ $\mathbbm{P}^{s,x}$ a.s.,
  so by definition of  quadratic variation we know that 
\begin{equation} \label{B2}
\underset{i< k}{\sum}\left( M_{t^k_{i+1}}^{t^k_i}\right)^2 \underset{k\rightarrow \infty}{\overset{\mathbbm{P}^{s,x}}{\longrightarrow}} [M^{s,x}]_u - [M^{s,x}]_t.
\end{equation}
In the sequel we will construct  an 
$\mathcal{F}_{t,u}$-measurable random variable $[M]^t_u$ such that for any $(s,x)\in[0,t]\times E$, 
$\sum_{i\leq k}\left( M_{t^k_{i+1}}^{t^k_i}\right)^2 \underset{k\rightarrow \infty}{\overset{\mathbbm{P}^{s,x}}{\longrightarrow}} [M]^t_u$. In that case
 $[M]^t_u$ will then be $\mathbbm{P}^{s,x}$ a.s. equal to $[M^{s,x}]_u - [M^{s,x}]_t$.
\\
\\
Let $x \in E$.
Since $M$ is a MAF, for any $k$, $\underset{i< k}{\sum}\left( M_{t^k_{i+1}}^{t^k_i}\right)^2$ is $\mathcal{F}_{t,u}$-measurable and therefore $\mathcal{F}^{t,x}_{t,u}$-measurable. Since $\mathcal{F}^{t,x}_{t,u}$ is complete, the limit in probability of this sequence, $[M^{t,x}]_u-[M^{t,x}]_t$, is still  $\mathcal{F}^{t,x}_{t,u}$-measurable. By Proposition \ref{Fversion}, there is  an $\mathcal{F}_{t,u}$-measurable variable which depends on $(t,x)$, 
that we call $a_t(x,\omega)$ such that
    \begin{equation} \label{B2bis}
 a_t(x,\omega) =  [M^{t,x}]_u-[M^{t,x}]_t, {\mathbb P}^{t,x} \ {\rm a.s.}
\end{equation}
We will show below that there is a jointly measurable version 
of 
$(x,\omega) \mapsto a_t(x,\omega)$.\\

For every integer $n\geq 0$, we set $a^n_t(x,\omega):=n\wedge a_t(x,\omega)$ 
which is in particular limit in probability of $n\wedge\underset{i\leq k}{\sum}\left( M_{t^k_{i+1}}^{t^k_i}\right)^2$ under $\mathbbm{P}^{t,x}$.
\\
For any integers $k,n$ and any $x\in E$, we define the finite positive measures $\mathbbm{Q}^{k,n,x}$, $\mathbbm{Q}^{n,x}$ and $\mathbbm{Q}^x$ on $(\Omega,\mathcal{F}_{t,u})$ by
\begin{enumerate}
\item $\mathbbm{Q}^{k,n,x}(F) := \mathbbm{E}^{t,x}\left[\mathds{1}_F\left(n\wedge\underset{i< k}{\sum}\left( M_{t^k_{i+1}}^{t^k_i}\right)^2\right)\right]$; \\
\item $\mathbbm{Q}^{n,x}(F): = \mathbbm{E}^{t,x}[\mathds{1}_F\left(a^n_t(x,\omega)\right)]$; \\
\item $\mathbbm{Q}^x(F) := \mathbbm{E}^{t,x}[\mathds{1}_F\left(a_t(x,\omega)\right)]$.
\end{enumerate}

When $k$ and $n$ are fixed, for any fixed $F$, by Proposition \ref{Borel}, 
\\
$x\longmapsto\mathbbm{E}^{t,x}\left[F\left(n\wedge\underset{i< k}{\sum}\left( M_{t^k_{i+1}}^{t^k_i}\right)^2\right)\right],$ is Borel. 

Then $n\wedge\underset{i< k}{\sum}\left( M_{t^k_{i+1}}^{t^k_i}\right)^2\overset{\mathbbm{P}^{t,x}}{\underset{k\rightarrow \infty}{\longrightarrow}} a^n_t(x,\omega)$, and this sequence is uniformly bounded by the constant $n$, so the convergence takes place in $L^1$, therefore 
$x\longmapsto\mathbbm{Q}^{n,x}(F)$ is also Borel as the pointwise limit in k of the functions $x\longmapsto\mathbbm{Q}^{k,n,x}(F)$.
 Similarly, $a^n_t(x,\omega)\underset{n\rightarrow \infty}{\overset{a.s.}{\longrightarrow}}a_t(x,\omega)$ and is non-decreasing, so by  monotone convergence theorem, being  a pointwise limit in $n$ of the functions
 $x\longmapsto\mathbbm{Q}^{n,x}(F)$, the function
$x\longmapsto\mathbbm{Q}^x(F)$ is Borel. 
We recall that $\mathcal{F}$ is separable.
%
The just two mentioned properties 
 and the fact that,
for any $x$, we also have (by item 3. above)
 $\mathbbm{Q}^{x}\ll \mathbbm{P}^{t,x}$, allows  to show
(see  Theorem 58 Chapter V in \cite{dellmeyerB})
the existence of a jointly measurable
 (for $\mathcal{B}(E)\otimes\mathcal{F}_{t,u}$) version
  of   $(x,\omega)\mapsto a_t(x,\omega)$, that we recall to be 
densities of $\mathbbm{Q}^{x}$ 
with respect to $\mathbbm{P}^{t,x}$. That version will still be denoted
by the same symbol.
\\
\\ 

We can now set $[M]^t_u(\omega)=a_t(X_t(\omega),\omega)$,
 which is a correctly defined $\mathcal{F}_{t,u}$-measurable random variable. For any $x$, since $\mathbbm{P}^{t,x}(X_t=x)=1$, we have
 the equalities 
\begin{equation} \label{B3bis}
[M]^t_u =a_t(x,\cdot)=[M^{t,x}]_u - [M^{t,x}]_t  \ \mathbbm{P}^{t,x} {\rm a.s.}
\end{equation}
We will moreover prove that
\begin{equation} \label{Etsx}
[M]^t_u = [M^{s,x}]_u - [M^{s,x}]_t\,\text{ }\,\mathbbm{P}^{s,x} \text{ a.s.},
\end{equation}
holds   for every $(s,x)\in[0,t]\times E$, and not just in the case $s=t$ 
that  we have just established in \eqref{B3bis}.
\\
\\
Let us fix $s < t$ 
 and $x \in E$.
We show that under any $\mathbbm{P}^{s,x}$, $[M]^t_u$ is the limit in probability of $\underset{i< k}{\sum}\left( M_{t^k_{i+1}}^{t^k_i}\right)^2$. Indeed, let $\epsilon>0$: the event $\left\{\left|\underset{i< k}{\sum}\left( M_{t^k_{i+1}}^{t^k_i}\right)^2-[M]^t_u\right|>\epsilon\right\}$ belongs to $\mathcal{F}_{t,T}$ so by conditioning and using the Markov property \eqref{Markov3} we have
\begin{equation*}
\begin{array}{rcl}
    &&\mathbbm{P}^{s,x}\left(\left|\underset{i< k}{\sum}\left( M_{t^k_{i+1}}^{t^k_i}\right)^2-[M]^t_u\right|>\epsilon\right)\\ &=& \mathbbm{E}^{s,x}\left[\mathbbm{P}^{s,x}\left(\left|\underset{i< k}{\sum}\left( M_{t^k_{i+1}}^{t^k_i}\right)^2-[M]^t_u\right|>\epsilon\middle|\mathcal{F}_t\right)\right]\\
    &=& \mathbbm{E}^{s,x}\left[\mathbbm{P}^{t,X_t}\left(\left|\underset{i< k}{\sum}\left( M_{t^k_{i+1}}^{t^k_i}\right)^2-[M]^t_u\right|>\epsilon\right)\right].
\end{array}
\end{equation*}
For any fixed $y$, 
by \eqref{B2} and \eqref{B3bis},
$\mathbbm{P}^{t,y}\left(\left|\underset{i< k}{\sum}\left( M_{t^k_{i+1}}^{t^k_i}\right)^2-[M]^t_u\right|>\epsilon\right)$ tends to zero when $k$ goes to infinity, 
 for every realization $\omega$, it yields
 $\mathbbm{P}^{t,X_t}\left(\left|\underset{i< k}{\sum}\left( M_{t^k_{i+1}}^{t^k_i}\right)^2-[M]^t_u\right|>\epsilon\right)$ tends  to zero when 
$k$ goes to infinity. Since this sequence is dominated by the constant $1$,
 that convergence still holds under the expectation 
with respect to the probability  the  probability $\mathbbm{P}^{s,x}$,
thanks to the dominated convergence theorem.
\\
\\
So we have built an $\mathcal{F}_{t,u}$-measurable variable $[M]^t_u$ such that under any $\mathbbm{P}^{s,x}$ with $s\leq t$, 
$[M^{s,x}]_u - [M^{s,x}]_t=[M]^t_u$ a.s. and this concludes the proof.
\end{proof}

We will now extend the result about quadratic variation to the angular bracket of MAFs. The next result can be seen as an extension of Theorem 15 Chapter XV in \cite{dellmeyerD} to a non-homogeneous context.

\begin{proposition}\label{AngleBracketAF}
Let $(B^t_u)_{(t,u)\in\Delta}$ be an increasing AF with $L^1$ terminal value, for any $(s,x)\in[0,T]\times E$, let 
$B^{s,x}$ be its cadlag version under $\mathbbm{P}^{s,x}$ and let $A^{s,x}$ be the predictable dual projection of $B^{s,x}$ in 
$(\Omega,\mathcal{F}^{s,x},(\mathcal{F}^{s,x}_t)_{t\in[0,T]},\mathbbm{P}^{s,x})$. Then there exists an increasing AF with $L^1$ terminal value $
(A^t_u)_{(t,u)\in\Delta}$ such that under any  $\mathbbm{P}^{s,x}$, the cadlag version of $A$ is $A^{s,x}$.
\end{proposition}

\begin{proof}
The first half of the demonstration will consist in showing that 
\begin{equation} \label{I422}
\forall  (s,x)\in[0,t]\times E,  \ 
(A^{s,x}_u - A^{s,x}_t) \ {\rm is} \  \mathcal{F}^{s,x}_{t,u}{\rm-measurable.}
\end{equation}
 \\
We start by recalling a property of the predictable dual projection which we will have to extend slightly. 
\\
Let us fix $(s,x)$ and the corresponding stochastic basis 
 $(\Omega,\mathcal{F}^{s,x},(\mathcal{F}^{s,x}_t)_{t\in[0,T]},\mathbbm{P}^{s,x})$. 
For any $F\in \mathcal{F}^{s,x}$, let $N^{s,x,F}$ be the cadlag version of the martingale, 
$r\longmapsto\mathbbm{E}^{s,x}[\mathds{1}_F|\mathcal{F}_r]$. Then for any $0\leq t\leq u \leq T$, the predictable projection of the process $r\mapsto \mathds{1}_F\mathds{1}_{[t,u[}(r)$ is $r\mapsto N^{s,x,F}_{r^-}\mathds{1}_{[t,u[}(r)$, see the proof of Theorem 43 Chapter VI in \cite{dellmeyerB}. Therefore by definition of the dual predictable projection (see Definition 73 Chapter VI in \cite{dellmeyerB}) we have
\begin{equation}\label{dualprojection}
\mathbbm{E}^{s,x}\left[\mathds{1}_F(A^{s,x}_u - A^{s,x}_t)\right]=\mathbbm{E}^{s,x}\left[\int_t^u N^{s,x,F}_{r^-}dB^{s,x}_r\right],
\end{equation}
for any $F\in\mathcal{F}^{s,x}$.
\\
\\
We will now prove some technical lemmas which in a sense extend this property, and will permit us to operate with a good common version of the random variable $\int_t^u N^{s,x,F}_{r^-}dB^{s,x}_r$ not depending on $(s,x)$. 
\\
\\
For the rest of the proof, $0\leq t<u\leq T$ will be fixed.

\begin{notation}\label{NFindep}
	Let $F\in\mathcal{F}_{t,T}$. We denote for any $r\in[t,T],\omega\in\Omega$, $N^F_r(\omega):=\mathbbm{P}^{t,X_t(\omega)}(F)$.
\end{notation}
It is clear that $N^F$ previously introduced is an $(\mathcal{F}_{t,r})_{r\in[t,T]}$-adapted process which does not depend on $(s,x)$, which takes values in $[0,1]$ for all $r,\omega$ and by Remark \ref{Rfuturefiltration}, for any $(s,x)\in[0,t]\times E$, $N^{s,x,F}$ is on $[t,T]$ a $\mathbbm{P}^{s,x}$-version of $N^F$.

\begin{lemma}\label{commonint}
Let  $F\in\mathcal{F}_{t,T}$.  There exists an $\mathcal{F}_{t,u}$-measurable random variable
which we will denote $\int_t^u N^F_{r^-}dB_r$ such that for any 
$(s,x)\in[0,t]\times E$, \\ $\int_t^u N^F_{r^-}dB_r=\int_t^u N^{s,x,F}_{r^-}dB^{s,x}_r$ $\mathbbm{P}^{s,x}$ a.s. 
\end{lemma}
\begin{remark} \label{Rsx}
By definition, the process $N^F$ introduced in Notation \ref{NFindep} and the r.v.
 $\int_t^u N^F_{r^-}dB_r$ will not depend on any $(s,x)$.
\end{remark}

\begin{proof}
In some sense we wish  to integrate $r\mapsto N^F_{r^-}$ against $B^t$
for fixed $\omega$. 
 However first we do not know a priori if the paths $r \mapsto N^F_r$ 
and $r \mapsto B^t_r$ are measurable, second $r \mapsto N^F_r$ 
may not have a left limit and $B^t$ may be not of bounded variation.
So it is not clear if $\int_t^u N^F_{r^-}dB^t_r$  makes sense
for any $\omega$.
 Moreover under a certain $\mathbbm{P}^{s,x}$, 
 $N^{F,s,x}$ and $B^{s,x}_{\cdot}- B^{s,x}_t$ are only versions of
$N^F$ and $B^t$  and  not indistinguishable to them.
Even if we could compute the overmentioned integral, it would not be clear if
 $\int_t^u N^F_{r^-}dB^t_r=\int_t^u N^{s,x,F}_{r^-}dB^{s,x}_r$ $\mathbbm{P}^{s,x}$ a.s.
\\
\\
We start by some considerations about $B$,
setting $W_{tu}:=\{\omega:\underset{r\in[t,u]\cap\mathbbm{Q}}{\text{sup }}B^t_r<\infty\}$ which is $\mathcal{F}_{t,u}$-measurable, and for $r\in[t,u]$
\begin{equation*}
\bar{B}^t_r(\omega) := \left\{\begin{array}{l}
\underset{\substack{t\leq v<r\\\ v\in\mathbbm{Q}}}{\text{sup }}B^t_v(\omega) \text{ if }\omega\in W_{tu}\\
0 \text{ otherwise}.
\end{array}\right.
\end{equation*}
$\bar{B}^t$  is an increasing, finite (for all $\omega$) process. 
In general, it is neither a measurable nor
an  adapted process; however for any $r\in[t,u]$, $\bar{B}^t_r$ is still  $\mathcal{F}_{t,u}$-measurable.
  Since it is increasing, it has right and left limits at each point
 for every $\omega$, so we
 can define  the process $\tilde{B}^t$ indexed on $[t,u]$ 
below:
\begin{equation} \label{EB10}
\tilde{B}^t_r := \underset{\substack{v\downarrow r\\\ v\in\mathbbm{Q}}}{\text{lim }} \bar{B}^t_v, r \in [t,u],
\end{equation}
when $u \in ]t,T[$ and
 $\tilde{B}^t_T :=B^t_T$ if $u=T$.
Therefore $\tilde{B}^t$ is an increasing, cadlag process. 
It is constituted by  $\mathcal{F}_{t,u}$-measurable random variables, 
and by Theorem 15 Chapter IV of \cite{dellmeyer75}, $\tilde{B}^t$ is a 
also a measurable process (indexed by   $[t,u]$).
\\
\\
We can show that $\tilde{B}^t$ is $\mathbbm{P}^{s,x}$-indistinguishable from $B^{s,x}_{\cdot}-B^{s,x}_t$ for any  
\\
$(s,x)\in[0,t]\times E$. Indeed, let
$(s,x)$ be fixed. Since $B^{s,x}_{\cdot}-B^{s,x}_t$ is a version of $B^t$ and 
 $\mathbbm{Q}$ being countable, there exists a  $\mathbbm{P}^{s,x}$-null set $\mathcal{N}$ such that for all $\omega\in\mathcal{N}^c$ and $r\in\mathbbm{Q}\cap[t,u]$, $B^{s,x}_r(\omega)-B^{s,x}_t(\omega)=B^t_r(\omega)$. Therefore for any $\omega\in\mathcal{N}^c$ and $r\in[t,u]$,
\begin{equation*}
\tilde{B}^t_r(\omega) = \underset{\substack{v\downarrow r\\\ v\in\mathbbm{Q}}}{\text{lim }}\underset{\substack{t\leq w<v\\\ w\in\mathbbm{Q}}}{\text{sup }}B^t_w(\omega) = \underset{\substack{v\downarrow r\\\ v\in\mathbbm{Q}}}{\text{lim }}\underset{\substack{t\leq w<v\\\ w\in\mathbbm{Q}}}{\text{sup }}B^{s,x}(\omega)_w-B^{s,x}(\omega)_t=B^{s,x}(\omega)_r-B^{s,x}(\omega)_t,
\end{equation*}
where the latter equality comes from the fact that $B^{s,x}(\omega)$ is cadlag
 and increasing. So we have constructed an increasing finite cadlag
(for all $\omega$) process and so the path
$ r \mapsto \tilde{B}^t(\omega)$ is a Lebesgue integrator on  $[t,u]$
for each  $\omega$.
\\
\\
We fix now $F\in\mathcal{F}_{t,T}$ and we discuss 
some issues related to  $N^F$ . Since it is positive, we can start 
defining the process $\bar{N}$, for index values $ r \in [t,T[$ by
  $\bar{N}^F_r:=\underset{\substack{v\downarrow r\\\ v\in\mathbbm{Q}}}{\text{liminf }}N^F_v,$ and setting $\bar{N}^F_T:=N^F_T$.
 This process is (by similar arguments as for $\tilde{B}^t$ defined 
in \eqref{EB10}), 
$\mathbbm{P}^{s,x}$-indistinguishable to $N^{s,x,F}$ for all $(s,x)\in[0,t]\times E$. 
For any $r\in [t,T]$, $N^F_r$ (see Notation \ref{NFindep}) is $\mathcal{F}_{t,r}$-measurable, 
 so $\bar{N}^F_r$ will also be  $\mathcal{F}_{t,r}$-measurable for any $r\in [t,T]$ by right-continuity of $\mathcal{F}_{t,\cdot}$ (see Definition \ref{canonicalspace}) .
 However, $\bar{N}^F$ is not necessarily cadlag for every $\omega$, and 
also not  necessarily a measurable process.
\\
\\
We subsequently define
\begin{equation*}
W'_{tu}:=\{\omega \in \Omega \vert \text{there exists a cadlag function }f \text{ such that }\bar{N}^F(\omega)=f\text{ on }[t,u]\cap\mathbbm{Q}\}.
\end{equation*}
By Theorem 18 b) in Chapter IV of \cite{dellmeyer75}, $W'_{tu}$ is $\mathcal{F}_{t,u}$-measurable so we can define on $[t,u]$ $\tilde{N}^F_r:=\bar{N}^F_r\mathds{1}_{W'_{tu}}$. $\tilde{N}^F$ is no longer $(\mathcal{F}_t)_{t\in[0,T]}$-adapted,
 however, it is now cadlag for all $\omega$ and therefore a measurable process by Theorem 15 Chapter IV of \cite{dellmeyer75}. The r.v.  $\tilde{N}^F_r$  
are still  $\mathcal{F}_{t,u}$-measurable 
 , and
 $\tilde{N}^F$ is still $\mathbbm{P}^{s,x}$-indistinguishable to $N^{s,x,F}$ on $[t,u]$ for any $(s,x)\in[0,t]\times E$.
\\
\\
Finally we can define $\int_t^u N^F_{r^-}dB_r:=\int_t^u \tilde{N}^F_{r^-}d\tilde{B}^t_r$ which is $\mathbbm{P}^{s,x}$ a.s. equal to $\int_t^u N^{s,x,F}_{r^-}dB^{s,x}_r$ for any $(s,x)\in[0,t]\times E$.

Moreover, since $\tilde{N}^F$ and $\tilde{B}$ are both measurable with respect to 
\\
$\mathcal{B}([t,u])\otimes\mathcal{F}_{t,u}$ 
, then  $\int_t^u N^F_{r^-}dB_r$ is $\mathcal{F}_{t,u}$-measurable.
\end{proof}

The lemma below is a conditional version of the property \eqref{dualprojection}.
\begin{lemma}\label{lemmabracketAF1}
For any $(s,x)\in[0,t]\times E$ and $F\in \mathcal{F}^{s,x}_{t,T}$ we have $\mathbbm{P}^{s,x}$-a.s.
\begin{equation*}
\mathbbm{E}^{s,x}\left[\mathds{1}_{F}(A^{s,x}_u - A^{s,x}_t)\middle|\mathcal{F}_t\right]=\mathbbm{E}^{s,x}\left[\int_t^u N^F_{r^-}dB_r\middle|\mathcal{F}_t\right].
\end{equation*}

\end{lemma}

\begin{proof}
Let $s,x,F$ be fixed. By definition of  conditional expectation, 
we need to show that for any $G\in\mathcal{F}_t$ we have
\begin{equation*}
\mathbbm{E}^{s,x}\left[\mathds{1}_G\mathds{1}_{F}(A^{s,x}_u - A^{s,x}_t)\right]=\mathbbm{E}^{s,x}\left[\mathds{1}_G\mathbbm{E}^{s,x}\left[\int_t^u N^{F}_{r^-}dB_r\middle|\mathcal{F}_t\right]\right] \text{ a.s.}
\end{equation*}
For $r\in[t,u]$ we have 
$\mathbbm{E}^{s,x}[\mathds{1}_{F\cap G}|\mathcal{F}_r]=\mathds{1}_G\mathbbm{E}^{s,x}[\mathds{1}_F|\mathcal{F}_r]$ a.s. therefore the cadlag versions of those
 processes are indistinguishable on $[t,u]$ and the random variables $\int_t^uN^{G\cap F}_{r^-}dB_r$ and $\mathds{1}_G\int_t^u N^{F}_{r^-}dB_r$ as defined in Lemma \ref{commonint} are a.s. equal. So by the non conditional property of dual predictable projection \eqref{dualprojection} we have
\begin{equation*}
\begin{array}{rcl}
    \mathbbm{E}^{s,x}\left[\mathds{1}_G\mathds{1}_{F}(A^{s,x}_u - A^{s,x}_t)\right]&=&\mathbbm{E}^{s,x}\left[\int_t^uN^{G\cap F}_{r^-}dB_r\right]\\
    &=&\mathbbm{E}^{s,x}\left[\mathds{1}_G\int_t^u N^{F}_{r^-}dB_r\right]\\
    &=&\mathbbm{E}^{s,x}\left[\mathds{1}_G\mathbbm{E}^{s,x}\left[\int_t^u N^{F}_{r^-}dB_r\middle|\mathcal{F}_t\right]\right],
\end{array}
\end{equation*}
which concludes the proof.
\end{proof}

\begin{lemma}\label{lemmabracketAF2}
For any $(s,x)\in[0,t]\times E$ and $F\in \mathcal{F}_{t,T}$ we have $\mathbbm{P}^{s,x}$-a.s.,
\begin{equation*}
\mathbbm{E}^{s,x}\left[\mathds{1}_{F}(A^{s,x}_u - A^{s,x}_t)\middle|\mathcal{F}_t\right]=\mathbbm{E}^{s,x}\left[\mathds{1}_{F}(A^{s,x}_u - A^{s,x}_t)\middle|X_t\right].
\end{equation*}
\end{lemma}
\begin{proof}
By Lemma \ref{lemmabracketAF1} we have 
\begin{equation*}
\mathbbm{E}^{s,x}\left[\mathds{1}_{F}(A^{s,x}_u - A^{s,x}_t)|\mathcal{F}_t\right]=\mathbbm{E}^{s,x}\left[\int_t^u N^F_{r^-}dB_r\middle|\mathcal{F}_t\right].
\end{equation*}
By Lemma \ref{commonint}, $\int_t^u N^F_{r^-}dB_r$ is $\mathcal{F}_{t,T}$ measurable so the Markov property \eqref{Markov3} implies 
\begin{equation*}
\mathbbm{E}^{s,x}\left[\int_t^u N^F_{r^-}dB_r\middle|\mathcal{F}_t\right] = \mathbbm{E}^{s,x}\left[\int_t^u N^F_{r^-}dB_r\middle|X_t\right],
\end{equation*}
therefore $\mathbbm{E}^{s,x}\left[\mathds{1}_{F}(A^{s,x}_u - A^{s,x}_t)|\mathcal{F}_t\right]$ is a.s. equal to a $\sigma(X_t)$-measurable r.v and so is a.s. equal to $\mathbbm{E}^{s,x}\left[\mathds{1}_{F}(A^{s,x}_u - A^{s,x}_t)|X_t\right].$
\end{proof}

We are now able to prove \eqref{I422} which is
 the first important issue of the
proof of Proposition \ref{AngleBracketAF}, which states 
that 
 By definition, a predictable dual projection 
is adapted
 so we already know that $(A^{s,x}_u - A^{s,x}_t)$ is $\mathcal{F}^{s,x}_{u}$-measurable, therefore by Remark \ref{RemFiltr}, it is enough to show that it is also $\mathcal{F}^{s,x}_{t,T}$-measurable.
\\
So we are going to show that
\begin{equation} \label{E422}
A^{s,x}_u - A^{s,x}_t = \mathbbm{E}^{s,x}\left[A^{s,x}_u - A^{s,x}_t|\mathcal{F}_{t,T}\right]\, \mathbbm{P}^{s,x}\text{ a.s.}
\end{equation}
For this we will show that 
\begin{equation}\label{EqAFbracket}
\mathbbm{E}^{s,x}\left[\mathds{1}_{F}(A^{s,x}_u - A^{s,x}_t)\right]=\mathbbm{E}^{s,x}\left[\mathds{1}_{F}\mathbbm{E}^{s,x}\left[A^{s,x}_u - A^{s,x}_t|\mathcal{F}_{t,T}\right]\right],
\end{equation}
for any $F \in {\mathcal F}$.
We will  prove \eqref{EqAFbracket} for 
$F\in\mathcal{F}$ event of the form $F=F_t\cap F_{t,T}$ 
with $F_t\in\mathcal{F}_t$ and $F_{t,T}\in\mathcal{F}_{t,T}$.
\\
 By item 4. of Remark \ref{RemFiltr}, such events form a $\pi$-system $\Pi$ which generates $\mathcal{F}$. 

Consequently,  by the monotone class theorem, \eqref{EqAFbracket}
  will remain true for any $F\in\mathcal{F}$ 
 and even in $\mathcal{F}^{s,x}$ since ${\mathbb P}^{s,x}$-null set
 will not impact the equality.  
This will imply \eqref{E422}
 so that $A^{s,x}_u - A^{s,x}_t$ is $\mathcal{F}^{s,x}_{t,T}$-measurable.
\\
At this point, as we have anticipated, we prove \eqref{EqAFbracket} 
for a fixed 
\\
 $F=F_t\cap F_{t,T}\in\Pi$.
 By Lemma \ref{lemmabracketAF2} we have
\begin{eqnarray*} 
\mathbbm{E}^{s,x}\left[\mathds{1}_{F}(A^{s,x}_u - A^{s,x}_t)\right]&=&\mathbbm{E}^{s,x}\left[\mathds{1}_{F_t}\mathbbm{E}^{s,x}\left[\mathds{1}_{F_{t,T}}(A^{s,x}_u - A^{s,x}_t)|\mathcal{F}_t\right]\right]\\
&=&\mathbbm{E}^{s,x}\left[\mathds{1}_{F_t}\mathbbm{E}^{s,x}\left[\mathds{1}_{F_{t,T}}(A^{s,x}_u - A^{s,x}_t)|X_t\right]\right] \\
&=& \mathbbm{E}^{s,x}\left[\mathds{1}_{F_t}\mathbbm{E}^{s,x}\left[\mathbbm{E}^{s,x}\left[\mathds{1}_{F_{t,T}}(A^{s,x}_u - A^{s,x}_t)|\mathcal{F}_{t,T}\right]|X_t\right]\right],
\end{eqnarray*}
where the latter equality holds since 
 $\sigma(X_t)\subset\mathcal{F}_{t,T}$.

Now since $\mathbbm{E}^{s,x}\left[\mathds{1}_{F_{t,T}}(A^{s,x}_u - A^{s,x}_t)|\mathcal{F}_{t,T}\right]$ is $\mathcal{F}_{t,T}$-measurable, 
the   Markov property \eqref{Markov3} allows us 
to substitute the  conditional $\sigma$-field $\sigma(X_t)$ 
with  $\mathcal{F}_t$ and obtain
\begin{eqnarray*} 
    \mathbbm{E}^{s,x}\left[\mathds{1}_{F}(A^{s,x}_u - A^{s,x}_t)\right] &=& 
 \mathbbm{E}^{s,x}\left[\mathds{1}_{F_t}\mathbbm{E}^{s,x}\left[\mathbbm{E}^{s,x}\left[\mathds{1}_{F_{t,T}}(A^{s,x}_u - A^{s,x}_t)|\mathcal{F}_{t,T}\right]|\mathcal{F}_t\right]\right]\\
    &=&\mathbbm{E}^{s,x}\left[\mathds{1}_{F_t}\mathbbm{E}^{s,x}\left[\mathds{1}_{F_{t,T}}(A^{s,x}_u - A^{s,x}_t)|\mathcal{F}_{t,T}\right]\right]\\
    &=&\mathbbm{E}^{s,x}\left[\mathds{1}_{F_t}\mathds{1}_{F_{t,T}}\mathbbm{E}^{s,x}\left[(A^{s,x}_u - A^{s,x}_t)|\mathcal{F}_{t,T}\right]\right]\\
    &=&\mathbbm{E}^{s,x}\left[\mathds{1}_{F}\mathbbm{E}^{s,x}\left[(A^{s,x}_u - A^{s,x}_t)|\mathcal{F}_{t,T}\right]\right].
\end{eqnarray*}
This concludes the proof of \eqref{EqAFbracket}, therefore   \eqref{E422}
 holds  so that  $A^{s,x}_u - A^{s,x}_t$ is $\mathcal{F}^{s,x}_{t,u}$-measurable
and so \eqref{I422} is established.
 This concludes the first part of the proof of Proposition \ref{AngleBracketAF}.
\\
We pass to the second part of the proof of Proposition \ref{AngleBracketAF} 
where  we will  show that for given $ 0 < t < u$ there is  
an $\mathcal{F}_{t,u}$-measurable r.v. $A^t_u$ 
such that for every $(s,x)\in[0,t]\times E$,  $(A^{s,x}_u-A^{s,x}_t)=A^t_u$ $\quad\mathbbm{P}^{s,x}$ a.s.
\\
\\
Similarly to what we did with the quadratic variation in 
Proposition  \ref{VarQuadAF}, we start by noticing that for any $x\in E$, since $(A^{t,x}_u-A^{t,x}_t)$ is $\mathcal{F}^{t,x}_{t,u}$-measurable, there exists by Proposition \ref{Fversion} an $\mathcal{F}_{t,u}$-measurable r.v. $a(x,\omega)$ such that 
\begin{equation}\label{E425}
	a(x,\omega)=A^{t,x}_u-A^{t,x}_t\quad \mathbbm{P}^{t,x}\text{ a.s.}
\end{equation}
As in the proof of Proposition  \ref{VarQuadAF}, we will show the existence of a jointly-measurable version of  $(x,\omega)\mapsto a(x,\omega)$.
\\
For  every $x\in E$  we define 
on $\mathcal{F}_{t,u}$ the positive measure
\begin{equation} \label{E424}
\mathbbm{Q}^x:F\longmapsto\mathbbm{E}^{t,x}\left[\mathds{1}_F(A^{t,x}_u - A^{t,x}_t)\right]=\mathbbm{E}^{t,x}\left[\mathds{1}_Fa(x,\omega)\right].
\end{equation}

By Lemma \ref{commonint},
 and \eqref{dualprojection}, for every
 $F\in \mathcal{F}_{t,u}$ we have
\begin{equation}\label{E426}
	\mathbbm{Q}^x(F)=\mathbbm{E}^{t,x}\left[\int_t^u N^F_{r^-}dB_r\right],
\end{equation}
 and we recall that $\int_t^u N^F_{r^-}dB_r$  does not depend on $x$. 
 So by Proposition \ref{Borel}  $x\longmapsto\mathbbm{Q}^x(F)$ is Borel for any $F$. Moreover, for any $x$,
 $\mathbbm{Q}^x\ll \mathbbm{P}^{t,x}$. Again by Theorem 58 Chapter V in \cite{dellmeyerB}, there exists a  version $(x,\omega)\mapsto a(x,\omega)$ 
measurable for $\mathcal{B}(E)\otimes\mathcal{F}_{t,u}$ of the 
related Radon-Nikodym densities.

 We can now set $A^t_u(\omega) := a(X_t(\omega),\omega)$ 
which is then an  $\mathcal{F}_{t,u}$-measurable r.v.  
\\
Since  $\mathbbm{P}^{t,x}(X_t=x)=1$ and \eqref{E425} hold, we have
\begin{equation}\label{E427}
	A^t_u = a(X_t,\cdot) = a(x,\cdot)= A^{t,x}_u - A^{t,x}_t\quad \mathbbm{P}^{t,x}\text{ a.s.}
\end{equation}

%

We now set $s<t$ and $x\in E$ and we want to show that  we still have 
\\
$A^t_u =  A^{s,x}_u - A^{s,x}_t$ $\mathbbm{P}^{s,x}$ a.s. So, as above, we
 consider  $F\in\mathcal{F}_{t,u}$ and, thanks to 
\eqref{dualprojection}
 we compute 
\\
\begin{equation}\label{E428}
 \begin{array}{rcl}
    \mathbbm{E}^{s,x}\left[\mathds{1}_F(A^{s,x}_u - A^{s,x}_t)\right] &=&\mathbbm{E}^{s,x}\left[\int_t^u N^F_{r^-}dB_r\right]\\
    &=&\mathbbm{E}^{s,x}\left[\mathbbm{E}^{s,x}\left[\int_t^u N^{F}_{r^-}dB_r|\mathcal{F}_t\right]\right]\\
    &=&\mathbbm{E}^{s,x}\left[\mathbbm{E}^{t,X_t}\left[\int_t^u N^F_{r^-}dB_r\right]\right]\\
    &=&\mathbbm{E}^{s,x}\left[\mathbbm{E}^{t,X_t}\left[\mathds{1}_FA^t_u\right]\right]\\
    &=&\mathbbm{E}^{s,x}\left[\mathbbm{E}^{s,x}\left[\mathds{1}_FA^t_u|\mathcal{F}_t\right]\right]\\
    &=&\mathbbm{E}^{s,x}\left[\mathds{1}_FA^t_u\right].
\end{array}
\end{equation}
Indeed, concerning the fourth equality we recall that, by \eqref{E424}, \eqref{E426} and 
\eqref{E427}, we have
$\mathbbm{E}^{t,x}\left[\int_t^u N^F_{r^-}dB_r\right]=\mathbbm{E}^{t,x}\left[\mathds{1}_FA^t_u\right]$ for all $x$, so this equality becomes an equality whatever random variable we plug into $x$. The third and fifth equalities come from the Markov property \eqref{Markov3} since $\int_t^u N^{F}_{r^-}dB_r$ and $A^t_u$ are $\mathcal{F}_{t,T}$-measurable.
\\
Then, adding $\mathbbm{P}^{s,x}$-null sets does not change the validity of \eqref{E428}, so we have for any $F\in\mathcal{F}^{s,x}_{t,u}$ that $\mathbbm{E}^{s,x}\left[\mathds{1}_F(A^{s,x}_u - A^{s,x}_t)\right]=\mathbbm{E}^{s,x}\left[\mathds{1}_FA^t_u\right]$.
\\
Finally, since we had shown in the first half of the proof that $A^{s,x}_u - A^{s,x}_t$ is $\mathcal{F}^{s,x}_{t,u}$-measurable, and since $A^t_u$ also has,
 by construction, the same measurability property, 
we can conclude that $A^{s,x}_u - A^{s,x}_t=A^t_u$  $\mathbbm{P}^{s,x}$ a.s.
\\
\\
Since this holds for every $t\leq u$ and $(s,x)\in[0,t]\times E$, $(A^t_u)_{(t,u)\in\Delta}$ is the desired AF, which ends the proof of 
Proposition \ref{AngleBracketAF}.
\end{proof}

\begin{corollary}\label{AFbracket}
Let $M$, $M'$ be two square integrable MAFs, let $M^{s,x}$ (respectively $M'^{s,x}$) be the cadlag version of $M$ (respectively $M'$) under $\mathbbm{P}^{s,x}$. 
Then there exists a bounded variation AF with $L^1$ terminal condition denoted $\langle M,M'\rangle$ such that under any $\mathbbm{P}^{s,x}$, the cadlag version of $\langle M,M'\rangle$ is $\langle M^{s,x},M'^{s,x}\rangle$. If $M=M'$ the AF $\langle M,M'\rangle$ will be denoted $\langle M\rangle$ and is increasing.
\end{corollary}
\begin{proof}
If $M=M'$, the corollary comes from the combination of Propositions \ref{VarQuadAF} and \ref{AngleBracketAF}, and the fact that the angular bracket of a square integrable martingale is the dual predictable projection of its quadratic variation. 
\\
Otherwise, it is clear that $M+M'$ and $M-M'$ are square integrable MAFs, so we can consider the increasing MAFs $\langle M-M'\rangle$ and $\langle M+M'\rangle$. We introduce the AF
\begin{equation*}
\langle M,M'\rangle = \frac{1}{4}(\langle M+M'\rangle - \langle M-M'\rangle),
\end{equation*}
which by polarization has cadlag version $\langle M^{s,x},M'^{s,x}\rangle$ under $\mathbbm{P}^{s,x}$. $\langle M,M'\rangle$ is therefore a  bounded variation AF with $L^1$ terminal condition.
\end{proof}

We are now going to study the Radon-Nikodym derivative of an increasing continuous AF with respect to 
some measure. The next result can be seen as an extension of Theorem 13 Chapter XV in \cite{dellmeyerD} in a non-homogeneous setup.

\begin{proposition}\label{RadonDerivAF}
Let $A$ be a positive, non-decreasing AF absolutely continuous with respect to some continuous non-decreasing function $V$, and for every $(s,x)\in[0,T[\times E$ let $A^{s,x}$ be the cadlag version of $A$ under $\mathbbm{P}^{s,x}$. There exists a Borel function $h\in\mathcal{B}([0,T]\times E,\mathbbm{R})$ such that for
every $(s,x)\in[0,T]\times E$,  $A^{s,x}=\int_s^{\cdot\vee s}h(r,X_r)dV_r$, in the sense of indistinguishability.
\end{proposition}

\begin{proof}
We set 
\begin{equation} \label{EDC} 
C^t_u=A^t_u +(V_u-V_t)+(u-t),
\end{equation}
 which is an AF with cadlag versions 
\begin{equation} \label{EDCsx} 
C^{s,x}_t=A^{s,x}_t+V_t+t,
\end{equation}
 and we start by showing the statement
for $A$ and $C$ instead of $A$ and $V$. We introduce the
 intermediary function $C$ so that for any $u>t$ that
 $\frac{A^{s,x}_u-A^{s,x}_t}{C^{s,x}_u-C^{s,x}_t}\in[0,1]$; that property  will 
be used extensively in connections with the application
of  dominated convergence theorem.
\\
\\ 
Since $A^{s,x}$ is non-decreasing for any $(s,x)\in[0,T]\times E$, $A$ can be taken positive (in the sense that $A^t_u(\omega)\geq 0$ for any $(t,u)\in\Delta$ and $\omega\in\Omega$) by considering $A^+$ (defined by $(A^+)^t_u(\omega):=A^t_u(\omega)^+$) instead of $A$.
\\
\\ 
For  $ t \in [0,T[$ we set   
\begin{align}\label{EB29}
K_t &:= \underset{n\rightarrow \infty}{\text{liminf }}\frac{A^t_{t+\frac{1}{n}}}{A^t_{t+\frac{1}{n}}+\frac{1}{n}+(V_{t+\frac{1}{n}}-V_t)} \nonumber \\
&=\underset{n\rightarrow \infty}{\text{lim }}\underset{p\geq n}{\text{inf }}\frac{A^t_{t+\frac{1}{p}}}{A^t_{t+\frac{1}{p}}+\frac{1}{p}+(V_{t+\frac{1}{p}}-V_t)}\\
&=\underset{n\rightarrow \infty}{\text{lim }}\underset{m\rightarrow \infty}{\text{lim  }}\underset{n\leq p \leq m}{\text{min }}\frac{A^t_{t+\frac{1}{p}}}{A^t_{t+\frac{1}{p}}+\frac{1}{p}+(V_{t+\frac{1}{p}}-V_t)}. \nonumber
\end{align}
By positivity, this liminf always exists and belongs to $[0,1]$ since the sequence belongs to $[0,1]$. For every $(s,x)\in[0,T]\times E$, since for all $t\geq s$ and $n\geq 0$, 
\\
$A^t_{t+\frac{1}{n}} = A^{s,x}_{t+\frac{1}{n}}-A^{s,x}_t$ $\mathbbm{P}^{s,x}$ a.s., then $K^{s,x}$ defined by
$K^{s,x}_t:=\underset{n\rightarrow \infty}{\text{liminf }}\frac{A^{s,x}_{t+\frac{1}{n}}-A^{s,x}_t}{C^{s,x}_{t+\frac{1}{n}}-C^{s,x}_t}$ is a $\mathbbm{P}^{s,x}$-version of $K$, for  $ t \in  [s,T[$.
\\
By  Lebesgue Differentiation theorem (see Theorem 12 Chapter XV in \cite{dellmeyerD} for a version of the theorem with a general atomless measure),
 for any $(s,x)$, for  $\mathbbm{P}^{s,x}$-almost all  $\omega$,
since
$ dC^{s,x}(\omega)$ is absolutely continuous with respect
to $dA^{s,x}(\omega)$, 
 $K^{s,x}(\omega)$ is a density of $dA^{s,x}(\omega)$ with respect to $dC^{s,x}(\omega)$. 
\\
\\
We now show that there exists a Borel function $k$ in $\mathcal{B}([0,T[\times E,\mathbbm{R})$
 such that under any $\mathbbm{P}^{s,x}$, $k(t,X_t)$ is on $[s,T[$ a version of $K$ (and therefore of $K^{s,x}$).\\
For every $t\in[0,T[$, $K_t$ is measurable with respect to $\underset{n\geq 0}{\bigcap} \mathcal{F}_{t,t+\frac{1}{n}}=\mathcal{F}_{t,t}$
by construction, taking into account Notation \ref{canonicalspace}.
  So for any $(t,x)\in[0,T]\times E$, by Proposition \ref{blumenthal},  there exists a constant which we denote $k(t,x)$ such that
\begin{equation} \label{EB31} 
K_t =  k(t,x),\quad  \mathbbm{P}^{t,x} {\rm a.s.} 
\end{equation}
For any integers $(n,m)$, we define $k^{n,m}$ by 
\begin{equation*}
(t,x)\mapsto\mathbbm{E}^{t,x}\left[\underset{n\leq p \leq m}{\text{min }}\frac{A^t_{t+\frac{1}{p}}}{A^t_{t+\frac{1}{p}}+\frac{1}{p}+(V_{t+\frac{1}{p}}-V_t)}\right],
\end{equation*}
and $k^n$ by 
\begin{equation} \label{EB33}
(t,x)\mapsto\mathbbm{E}^{t,x}\left[\underset{p\geq n}{\text{inf }}\frac{A^t_{t+\frac{1}{p}}}{A^t_{t+\frac{1}{p}}+\frac{1}{p}+(V_{t+\frac{1}{p}}-V_t)}\right],
\end{equation}
We start  showing that $\tilde{k}^{n,m}$ defined by
\begin{equation} \label{EB32}
\begin{array}{rcl}
    (s,x,t)&\longmapsto& \mathbbm{E}^{s\wedge t,x}\left[\underset{n\leq p\leq m}{\text{min }}\frac{A^t_{t+\frac{1}{p}}}{A^t_{t+\frac{1}{p}}+\frac{1}{p}+(V_{t+\frac{1}{p}}-V_t)}\right],\\ \relax 
    [0,T]\times E\times[0,T[  &\longrightarrow& [0,1],
\end{array}
\end{equation}
is jointly Borel. 
\\ 
If we fix $t$, then by Proposition \ref{Borel} $(s,x)\longmapsto\mathbbm{E}^{s,x}\left[\underset{n\leq p\leq m }{\text{min }}\frac{A^t_{t+\frac{1}{p}}}{A^t_{t+\frac{1}{p}}+\frac{1}{p}+(V_{t+\frac{1}{p}}-V_t)}\right]$ is a Borel function, so by composing with $(s,x)\mapsto (s\wedge t,x)$, then 
\\
$(s,x)\mapsto \tilde{k}^{n,m}(s,x,t)$ is Borel. Moreover, if we fix $(s,x)\in [0,T[\times E$ we  show below that $t\mapsto \tilde{k}^{n,m}(s,x,t)$ is continuous, which by Lemma 4.51 in \cite{aliprantis} implies the joint measurability of $\tilde{k}^{n,m}$.
\\

To show that mentioned continuity property, we first remark that
 $\tilde{k}^{n,m}(s,x,\cdot)$ is constant on $[0,s]$; moreover
 $A^{s,x}$ is continuous ${\mathbb P}^{s,x}$ a.s. $V$ is continuous,
 and the minimum of a finite number of continuous
functions remains continuous.
 Let  $t_q\underset{q \rightarrow \infty}{\longrightarrow} t$
be a converging sequence  in $[s,T[$.
 Then
  $\underset{n\leq p\leq m }{\text{min }}\frac{A^{s,x}_{t_q+\frac{1}{p}}-A^{s,x}_{t_q}}{A^{s,x}_{t_q+\frac{1}{p}}-A^{s,x}_{t_q}+\frac{1}{p}+(V_{t_q+\frac{1}{p}}-V_{t_q})}$ tends a.s. to $\underset{n\leq p\leq m }{\text{min }}\frac{A^{s,x}_{t+\frac{1}{p}}-A^{s,x}_t}{A^{s,x}_{t+\frac{1}{p}}-A^{s,x}_t+\frac{1}{p}+(V_{t+\frac{1}{p}}-V_t)}$ when $q$ tends to infinity. Since for any $s\leq t\leq u$, $A^t_u= A^{s,x}_u-A^{s,x}_t$ $\mathbbm{P}^{s,x}$ a.s., then $\frac{A^{t_q}_{t_q+\frac{1}{p}}}{A^{t_q}_{t_q+\frac{1}{p}}+\frac{1}{p}+(V_{t_q+\frac{1}{p}}-V_{t_q})}$ tends a.s. to $\frac{A^t_{t+\frac{1}{p}}}{A^t_{t+\frac{1}{p}}+\frac{1}{p}+(V_{t+\frac{1}{p}}-V_t)}$.  All those terms being
 smaller than one, by dominated convergence theorem, the mentioned convergence
also  holds under the expectation, hence the announced continuity
related to $\tilde k^{n,m}$ is established.  
 \\
 \\
Since $k^{n,m}(t,y)=\tilde{k}^{n,m}(t,t,y)$, by composition we can deduce that for any $n,m$, $k^{n,m}$ is Borel. By the dominated convergence theorem, $k^{n,m}$ tends pointwise to $k^n$ (which was defined in \eqref{EB33},
  when $m$ goes to infinity so  $k^n$ are also Borel
for every $n$. 
Finally, keeping in mind \eqref{EB29} nd \eqref{EB31}
 we have $\mathbbm{P}^{t,x}$ a.s. 
\begin{equation*}
k(t,x) = K_t = \underset{n\rightarrow \infty}{\text{lim }}\underset{p\geq n}{\text{inf }}\frac{A^t_{t+\frac{1}{p}}}{A^t_{t+\frac{1}{p}}+\frac{1}{p}+(V_{t+\frac{1}{p}}-V_t)}.
\end{equation*}
Taking the expectation and again by the dominated convergence theorem, $k^n$
(defined in \eqref{EB33}) tends pointwise to $k$ when $n$ goes to infinity so $k$ is Borel.
\\
\\
We now show that, for any $(s,x)\in [0,T]\times E$,  $k(\cdot,X_{\cdot})$ is a $\mathbbm{P}^{s,x}$-version of $K$ on $[s,T[$. 
\\
Since $\mathbbm{P}^{t,x}(X_t=x)=1$, we know that for any $t \in [0,T]$, $x\in E$, we have 
$K_t=k(t,x)=k(t,X_t)$ $\mathbbm{P}^{t,x}$-a.s., and we prove below
 that for any $t \in [0,T]$, $(s,x)\in [0,t]\times E$, we have 
$K_t=k(t,X_t)$ $\mathbbm{P}^{s,x}$-a.s.
\\ 
Let $t \in [0,T]$ be fixed.
Since $A$ is an AF, for any $n$, $\frac{A^t_{t+\frac{1}{p}}}{A^t_{t+\frac{1}{p}}+\frac{1}{n}+(V_{t+\frac{1}{n}}-V_t)}$ is $\mathcal{F}_{t,t+\frac{1}{n}}$-measurable.
\\
So the event 
$\left\{\underset{n\rightarrow \infty}{\text{liminf }}\frac{A^t_{t+\frac{1}{n}}}{A^t_{t+\frac{1}{n}}+\frac{1}{n}+(V_{t+\frac{1}{n}}-V_t)}=k(t,X_t)\right\}$ belongs to $\mathcal{F}_{t,T}$ and
 by Markov property \eqref{Markov3}, for any $(s,x)\in[0,t]\times E$, 
we get 
\begin{eqnarray*}
\mathbbm{P}^{s,x}(K_t= k(t,X_t))&=&\mathbbm{E}^{s,x}[\mathbbm{P}^{s,x}\left(K_t= k(t,X_t)\middle|\mathcal{F}_t\right)]  \nonumber \\
 &=&\mathbbm{E}^{s,x}[\mathbbm{P}^{t,X_t}\left(K_t= k(t,X_t)\right)]\\
 &=& 1. \nonumber
\end{eqnarray*}

For any $(s,x)$, the process $k(\cdot,X_{\cdot})$ is therefore on $[s,T[$ a 
$\mathbbm{P}^{s,x}$-modification of $K$ 
 and therefore of $K^{s,x}$.
However it is not yet clear if provides another density of $dA^{s,x}$
 with respect to $dC^{s,x}$, which was defined at \eqref{EDCsx}.
\\
Considering that $(t,u,\omega)\mapsto V_u-V_t$ also defines a positive 
non-decreasing AF absolutely continuous with respect to $C$,
defined in \eqref{EDC},
 we proceed similarly as at the beginning of the proof,
 replacing the AF $A$ with $V$. 
\\ 
\\
Let the process $K'$ be defined by
\begin{equation*}
K'_t = \underset{n\rightarrow \infty}{\text{liminf }}\frac{V_{t+\frac{1}{n}}-V_t}{A^t_{t+\frac{1}{n}}+\frac{1}{n}+(V_{t+\frac{1}{n}}-V_t)},
\end{equation*}
and for any$(s,x)$, let $K'^{s,x}$ be defined on $[s,T[$ by
\begin{equation*}
K'^{s,x}_t = \underset{n\rightarrow \infty}{\text{liminf }}\frac{V_{t+\frac{1}{n}}-V_t}{A^{s,x}_{t+\frac{1}{n}}-A^{s,x}_t+\frac{1}{n}+(V_{t+\frac{1}{n}}-V_t)}.
\end{equation*}
Then, for any $(s,x)$, $K'^{s,x}$ on $[s,T[$ is
 a $\mathbbm{P}^{s,x}$-version of $K'$, and 
it constitutes  a density of $dV(\omega)$ with respect to $dC^{s,x}(\omega)$ on $[s,T[$, for almost all $\omega$. One shows then the  existence of
 a Borel function $k'$ such that for any $(s,x)$, $k'(\cdot,X_{\cdot})$ is a $\mathbbm{P}^{s,x}$-version of $K'$ and a modification 
of $K'^{s,x}$ on $[s,T[$.
\\
So for any $(s,x)$,   under $\mathbbm{P}^{s,x},$ we can write

\begin{equation*}
\left\{\begin{array}{rcl}
A^{s,x}&=&\int_s^{\cdot\vee s}K^{s,x}_rdC^{s,x}_r \\
V_{\cdot\vee s} - V_s &=& \int_s^{\cdot\vee s} K'^{s,x}_rdC^{s,x}_r
\end{array}\right.
\end{equation*}
Now since $dA^{s,x}\ll dV$, for a fixed $\omega$, the set
 $\{r \in [s,T]  |K'^{s,x}_r(\omega)=0\}$ is negligible with respect to 
 $dV$ so also for $dA^{s,x}(\omega)$ and therefore
 we can write
\begin{equation*}
\begin{array}{rcl}
A^{s,x}&=&\int_s^{\cdot\vee s}K^{s,x}_rdC^{s,x}_r \\
&=&\int_s^{\cdot\vee s}\frac{K^{s,x}_r}{ K'^{s,x}_r}
\mathds{1}_{\{K'^{s,x}_r\neq 0\}}K'^{s,x}_rdC^{s,x}_r \\ 
&+& \int_s^{\cdot\vee s}   \mathds{1}_{\{K'^{s,x}_r = 0\}}  dA^{s,x}_r 
  \\ 
&=&\int_s^{\cdot\vee s}\frac{K^{s,x}_r}{K'^{s,x}_r}\mathds{1}_{\{K'^{s,x}_r\neq 0\}}dV_r,
\end{array}
\end{equation*}
where we use the convention  that
for any two functions $\phi,\psi$ then $\frac{\phi}{\psi}\mathds{1}_{\psi\neq 0}$ is defined by by
\begin{equation*}
\frac{\phi}{\psi}\mathds{1}_{\{\psi\neq 0\}}(x)=\left\{\begin{array}{l}
\frac{\phi(x)}{\psi(x)}\text{ if } \psi(x)\neq 0\\
0\text{ if } \psi(x)=0.
\end{array}\right.
\end{equation*}
We set now
$h:=\frac{k}{k'}\mathds{1}_{\{k'_r\neq 0\}}$ which is Borel, and clearly for any $(s,x)$, $h(t,X_t)$ is a $\mathbbm{P}^{s,x}$-version of $H^{s,x}:=\frac{K^{s,x}}{K'^{s,x}}\mathds{1}_{\{K'^{s,x}\neq 0\}}$ on $[s,T[$. So by Lemma 5.12 in \cite{paper1preprint},  \\
$H^{s,x}_t=h(t,X_t)$ $dV\otimes d\mathbbm{P}^{s,x}$ a.e. 
and finally we have shown that under any $\mathbbm{P}^{s,x}$, 
\\
$A^{s,x}=\int_s^{\cdot\vee s}h(r,X_r)dV_r$ on $[0,T[$.
Without change of notations we extend $h$ to $[0,T] \times E$
by zero for $t = T$.
Since $A^{s,x}$ is continuous  $\mathbbm{P}^{s,x}$-a.s.
previous equality extends to $T$.

\end{proof}

\begin{proposition}\label{VarTotAF}
	Let $(A^t_u)_{(t,u)\in\Delta}$ be an AF with bounded variation and taking $L^1$ values. 
	Then there exists an increasing AF which we denote $(Pos(A)^t_u)_{(t,u)\in\Delta}$ (resp. $(Neg(A)^t_u)_{(t,u)\in\Delta}$ ) and which, for any $(s,x)\in[0,T]\times E$, has  	$Pos(A^{s,x})$ (resp. $Neg(A^{s,x})$)
	as cadlag version
	under $\mathbbm{P}^{s,x}$. 
\end{proposition}

\begin{proof}
By definition of the total variation of a  bounded variation function,
 the following holds. For every $(s,x)\in[0,T]\times E$,
  $s\leq t\leq u \leq T$ for $\mathbbm{P}^{s,x}$ almost all 
  $\omega\in\Omega$,
and any sequence of subdivisions of $[t,u]$: $t=t^k_1<t^k_2<\cdots<t^k_k=u$ such that $\underset{i<k}{\text{min }}(t^k_{i+1}-t^k_i)
\underset{k\rightarrow \infty}{\longrightarrow} 0$ we have 
\begin{equation}
\underset{i< k}{\sum} |A^{s,x}_{t^k_{i+1}}(\omega) - A^{s,x}_{t^k_i}(\omega)| \underset{k\rightarrow \infty}{\longrightarrow}Var(A^{s,x})_u(\omega)-Var(A^{s,x})_t(\omega),
\end{equation}
taking into account the considerations of the end of Section \ref{SPrelim}.
By Proposition 3.3 in \cite{jacod} Chapter I, 
we have $Pos(A^{s,x})=\frac{1}{2}(Var(A^{s,x})+A^{s,x})$ and $Neg(A^{s,x})=\frac{1}{2}(Var(A^{s,x})-A^{s,x})$. Moreover, for any $x\in\mathbbm{R}$ we know that $x^+=\frac{1}{2}(|x|+x)$ and $x^-=\frac{1}{2}(|x|-x)$, so we also have
\begin{equation}
\left\{
\begin{array}{rcl}
\underset{i< k}{\sum} (A^{s,x}_{t^k_{i+1}}(\omega) - A^{s,x}_{t^k_i}(\omega))^+ &\underset{k\rightarrow \infty}{\longrightarrow}&Pos(A^{s,x})_u(\omega)-Pos(A^{s,x})_t(\omega)\\
\underset{i< k}{\sum} (A^{s,x}_{t^k_{i+1}}(\omega) - A^{s,x}_{t^k_i}(\omega))^- &\underset{k\rightarrow \infty}{\longrightarrow}&Neg(A^{s,x})_u(\omega)-Neg(A^{s,x})_t(\omega),
\end{array}
\right.
\end{equation}
for  $\mathbbm{P}^{s,x}$ almost all $\omega$.
Since the convergence a.s. implies the convergence in probability, 
 for every $(s,x)\in[0,T]\times E$,  $s\leq t\leq u$ and any sequence of subdivisions of $[t,u]$: $t=t^k_1<t^k_2<\cdots<t^k_k=u$ such that
 $\underset{i<k}{\text{min }}(t^k_{i+1}-t^k_i)
\underset{k\rightarrow \infty}{\longrightarrow} 0$, we have
\begin{equation}
\left\{
\begin{array}{rcl}
\underset{i< k}{\sum}\left( A_{t^k_{i+1}}^{t^k_i}\right)^+ &\underset{k\rightarrow \infty}{\overset{\mathbbm{P}^{s,x}}{\longrightarrow}}& Pos(A^{s,x})_u - Pos(A^{s,x})_t\\
\underset{i< k}{\sum}\left( A_{t^k_{i+1}}^{t^k_i}\right)^- &\underset{k\rightarrow \infty}{\overset{\mathbbm{P}^{s,x}}{\longrightarrow}}& Neg(A^{s,x})_u - Neg(A^{s,x})_t.
\end{array}
\right.
\end{equation}
The proof can now be performed according to the same arguments as in the proof of Proposition \ref{VarQuadAF}, replacing $M$ with $A$, the quadratic increments with the positive (resp. negative) increments, and the quadratic variation with the positive (resp. negative) variation of an adapted process.
\end{proof}

We recall a definition and a result from \cite{paper1preprint}. We assume for now that we are given a fixed stochastic basis fulfilling the usual conditions, and a non-decreasing function $V$.
\begin{notation}\label{H2V}
We denote $\mathcal{H}^{2,V} := \{M\in\mathcal{H}^2_0|d\langle M\rangle \ll dV\}$ and $\mathcal{H}^{2,\perp V} := \{M\in\mathcal{H}^2_0|d\langle M\rangle \perp dV\}$.
\end{notation}
Proposition 3.6 in  \cite{paper1preprint} states the following.  
\begin{proposition}\label{DecompoMart}
$\mathcal{H}^{2,V}$ and $\mathcal{H}^{2,\perp V}$ are orthogonal sub-Hilbert spaces of $\mathcal{H}^2_0$ and $\mathcal{H}^2_0 = \mathcal{H}^{2,V}\oplus^{\perp}\mathcal{H}^{2,\perp V}$. Moreover, any element of $\mathcal{H}^{2,V}_{loc}$ 
 is strongly orthogonal to any element of $\mathcal{H}^{2,\perp V}_{loc}$.
\end{proposition}
For any $M\in\mathcal{H}^2_0$, we denote by $M^V$ its projection on $\mathcal{H}^{2,V}$.
\\
\\
We can now finally establish the main result of the present note.

\begin{proposition}\label{BracketMAFnew}
	Let $V$ be a continuous non-decreasing function.
	Let $M,N$ be two square integrable MAFs, and assume that the AF $\langle N\rangle$ is absolutely continuous with respect to $V$.
	There exists a function $v\in\mathcal{B}([0,T]\times E,\mathbbm{R})$ such that for any $(s,x)$, $\langle M^{s,x}, N^{s,x}\rangle =\int_s^{\cdot\vee s}v(r,X_r) V_r$.
\end{proposition}
\begin{proof}
	By Corollary \ref{AFbracket}, there exists a bounded variation AF with $L^1$ values denoted $\langle M,N\rangle$ such that under any $\mathbbm{P}^{s,x}$, the cadlag version of $\langle M,N\rangle$ is $\langle M^{s,x},N^{s,x}\rangle$.
	\\
	By Proposition \ref{VarTotAF}, there exists an increasing AF with $L^1$ values denoted $Pos(\langle M,N\rangle)$ (resp. $Neg(\langle M,N\rangle)$)  such that under any $\mathbbm{P}^{s,x}$, the cadlag version of $Pos(\langle M,N\rangle)$ (resp. $Neg(\langle M,N\rangle)$) is $Pos(\langle M^{s,x},N^{s,x}\rangle)$ (resp. $Neg(\langle M^{s,x},N^{s,x}\rangle)$).
	\\
	We fix some $(s,x)$ and the associated probability $\mathbbm{P}^{s,x}$. Since $\langle N\rangle$ is absolutely continuous with respect to $V$, comparing Definition \ref{DefAF} and Notation \ref{H2V}
 we have   $N^{s,x}\in\mathcal{H}^{2,V}$. Therefore 
by Proposition \ref{DecompoMart} we have
	\begin{equation}
		\begin{array}{rcl}
		\langle M^{s,x},N^{s,x}\rangle &=& \langle (M^{s,x})^V,N^{s,x}\rangle\\
		&=&\frac{1}{4}\langle (M^{s,x})^V+N^{s,x}\rangle - \frac{1}{4}\langle (M^{s,x})^V-N^{s,x}\rangle.
		\end{array}
	\end{equation}
	Since both processes $\frac{1}{4}\langle (M^{s,x})^V+N^{s,x}\rangle$, $\frac{1}{4}\langle (M^{s,x})^V-N^{s,x}\rangle$ are increasing and starting at zero, we have $Pos(\langle M^{s,x},N^{s,x}\rangle)=\frac{1}{4}\langle (M^{s,x})^V+N^{s,x}\rangle$ and $Neg(\langle M^{s,x},N^{s,x}\rangle)=\frac{1}{4}\langle (M^{s,x})^V-N^{s,x}\rangle$. Now since $(M^{s,x})^V+N^{s,x}$ and $(M^{s,x})^V-N^{s,x}$ belong to $\mathcal{H}^{2,V}$, we have shown that $dPos(\langle M^{s,x},N^{s,x}\rangle)\ll dV$ and $dNeg(\langle M^{s,x},N^{s,x}\rangle)\ll dV$ in the sense of stochastic measures.
	\\
	Since this holds for all $(s,x)$  Proposition \ref{RadonDerivAF}
insures
 the existence of two functions $v_+,v_-$ in $\mathcal{B}([0,T]\times E,\mathbbm{R})$ such that for any $(s,x)$, $Pos(\langle M^{s,x}, N^{s,x}\rangle) =\int_s^{\cdot\vee s}v_+(r,X_r)dV_r$ and $Neg(\langle M^{s,x}, N^{s,x}\rangle) =\int_s^{\cdot\vee s}v_-(r,X_r)dV_r$.
	\\
	The conclusion now follows setting $v=v_+-v_-$.
\end{proof}

\bibliographystyle{plain}

\bibliography{biblioPhDBarrasso}

\end{document}